\documentclass{amsart}
\usepackage{amsmath,amssymb,amsthm}
\usepackage{graphicx}
\newtheorem{Def}{Definition}[section]
\newtheorem{lemma}[Def]{Lemma}
\newtheorem{prop}[Def]{Proposition}
\newtheorem{thm}[Def]{Theorem}

\newtheorem{ex}[Def]{Example}
\newtheorem{rmk}[Def]{Remark}
\newcommand{\transv}{\mathrel{\text{\tpitchfork}}}
\makeatletter
\newcommand{\tpitchfork}{%
  \vbox{
    \baselineskip\z@skip
    \lineskip-.52ex
    \lineskiplimit\maxdimen
    \m@th
    \ialign{##\crcr\hidewidth\smash{$-$}\hidewidth\crcr$\pitchfork$\crcr}
  }%
}

\makeatletter
\@namedef{subjclassname@2020}{\textup{2020} Mathematics Subject Classification}
\makeatother

\begin{document}
\title[Diagrammatic characterization of spatial surfaces]{A diagrammatic presentation and its characterization\\of non-split compact surfaces in the 3-sphere}
\author[S.~Matsuzaki]{Shosaku Matsuzaki}
\address[S.~Matsuzaki]{Liberal Arts Education Center, Ashikaga University, 268-1 Ohmae-cho, Ashikaga-shi, Tochigi 326-8558, Japan}
\email{matsuzaki.shosaku@g.ashikaga.ac.jp}

\maketitle

\begin{abstract}
We give a presentation for a non-split compact surface embedded in the 3-sphere $S^3$ by using diagrams of spatial trivalent graphs equipped with signs and we define  Reidemeister moves for such signed diagrams.
We show that two diagrams of embedded surfaces are related by Reidemeister moves if and only if the surfaces represented by the diagrams are ambient isotopic in $S^3$.
\end{abstract}

\section{Introduction}\label{SEC:introduction}
A fundamental problem in Knot Theory is to classify knots and links up to ambient isotopy in $S^3$.
Two knots are equivalent if and only if diagrams of them are related by Reidemeister moves~\cite{r1927}:
Reidemeister moves characterize a combinatorial structure of knots.
Diagrammatic characterizations for spatial graphs, handlebody-knots and surface-knots are also known, where a spatial graph is a graph in $S^3$, a handlebody-knot is a handlebody in $S^3$, and a surface-knot is a closed surface in $S^4$ (cf.~\cite{ishii2008}, \cite{kauffman1989}, \cite{roseman1998}).
We often use invariants to distinguish the above knots.
Many invariants have been discovered on the basis of diagrammatic characterizations.
In this paper, we consider presentation of a compact surface embedded in $S^3$, which we call a spatial surface.
For a knot or link, we immediately obtain its diagram by perturbing the z-axis of projection slightly.
For a spatial surface, however, perturbing the spatial surface is not enough to present it in a useful form: it may be overlapped and folded complexly by multiple layers in the direction of the z-axis of $\mathbb R^3\subset S^3$.
We will give a diagram for spatial surfaces by using a trivalent spine equipped with information of twisted bands. 
\par
If any component of a spatial surface has non-empty boundary, we take a trivalent spine of the surface and take a thin regular neighborhood of the spine;
a regular neighborhood of a spine is equivalent to the original spatial surface.
In stead of the original surface,
we consider the regular neighborhood by using a spatial trivalent graph diagram equipped with information of twisted bands.
In Section~\ref{SEC:spatial_surface}, we give a characterization for spatial surfaces with boundary (Theorems~\ref{THM:main_non_oriented} and \ref{THM:main_oriented}).
In the next paper~\cite{IMM}, we constructed a coloring invariant of oriented spatial surfaces by using Theorem~\ref{THM:main_oriented}, and distinguished some pairs of oriented spatial surfaces.
The proofs of Theorems~\ref{THM:main_non_oriented} and \ref{THM:main_oriented} are written in Section~\ref{SEC:proof}; we need delicate consideration to avoid difficulty and complexity about information of twisted bands in spatial surfaces.
In the process of showing main theorems, we will give a characterization of trivalent spines of surfaces (Theorem~\ref{THM:Luo}).
\par
When we consider a non-split spatial surface that has closed components, we remove an open disk from each closed component of it;
then, we have a spatial surface with boundary.
The spatial surface with boundary loses no information up to ambient isotopy after removing an open disk (Proposition~\ref{PROP:many_hole_reconstruct_unique}).
Therefore, it is sufficient to consider a spatial surface that has boundary when we consider non-split spatial surfaces.
We can see some studies for closed surfaces in $S^3$
in~\cite{homma1954}, \cite{suzuki1974} and \cite{tsukui1970}.
Homma defined an unknotted polygon, which is a non-splittable loop of a closed surface embedded in $S^3$, and showed that every closed surface in $S^3$ has an unknotted polygon in~\cite{homma1954}.
On the base of this fact, Suzuki defined a complexity for a closed surface embedded in 3-manifold and studied it in~\cite{suzuki1974}.
Tsukui showed the uniqueness of decompositions for closed genus 2 surfaces in $\mathbb R^3$ in~\cite{tsukui1970}.
In those studies, however, we directly deal with closed surfaces without using something like a diagram.
We expect developments of those studies by using diagrams of spatial surfaces.
On the other hand, we can regard a knot, link or handlebody-knot as a spatial closed surface.
This suggests that we can systematically study knots, links and handlebody-knots by the new framework of spatial surfaces (Section~\ref{SEC:non_split_spatial_surface}).
We have a new diagrammatic characterization of knots (Theorem~\ref{THM_knot_new_Reidemeister_move}).
\par
Afterwards I knew that Theorem~\ref{THM:main_non_oriented} is a corollary of Proposition~1.3.8 in~\cite{bp2012}.
In the former part of the paper, the authors study a ribbon surface, which is a compact surface $F$ with non-empty boundary embedded in the 4-ball $B^4$ such that $\partial F$ is contained in $\partial B^4=S^3$.
Proposition~1.3.8 is proved on the Morse theory, while Theorem~\ref{THM:main_non_oriented} in this paper is shown by using a 3-dimensional way.

\section{An IH-move for trivalent spines on a surface with boundary}\label{SEC:Luo}
We prepare some notations used throughout this paper.
We denote by $\operatorname{\#}S$ the cardinality of a set $S$. 
We denote by $\mbox{cl}_X(U)$ the closure of a subset $U$ of a topological space $X$.
We denote by $\partial M$ and $\operatorname{int}M$ the boundary and the interior of a topological manifold $M$, respectively.
\par
We assume that a graph is finite, which has finite edges and vertices.
A graph is {\it trivalent} if every vertex of it is trivalent.
A trivalent graph may have a connected component that has no vertices, that is, a trivalent graph may have a circle component.
A {\it surface with boundary} is a compact surface such that every component of it has a non-empty boundary.
For a surface $F$ with boundary, a graph $G$ in $F\setminus N_{\partial F}$ is a {\it spine} of $F$ if $N_G$ and $\operatorname{cl}_F(F\setminus N_{\partial F})$ are ambient isotopic in $F$, where $N_G$ and $N_{\partial F}$ mean regular neighborhoods of $G$ and $\partial F$ in $F$, respectively.
In this section, we suppose that a surface with boundary has exactly one component.
\par
A disk has no trivalent spines, and an annulus or a M\"{o}bius band has exactly one trivalent spine up to ambient isotopy, which is a circle.
Here we remember that a circle is regarded as a trivalent graph.
However, other surfaces with boundary have infinitely many trivalent spines up to ambient isotopy.
Theorem~\ref{THM:Luo} claims that these spines are related by finitely many {\it IH-moves}, see Fig.~\ref{FIG:definition_IH_move}.
Our goal in Section~\ref{SEC:Luo} is to prove Theorem~\ref{THM:Luo}, which gives a characterization of trivalent spines of a surface with boundary.
		\begin{figure}[htbp]\centerline{
		\includegraphics{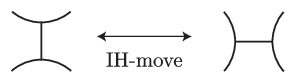}}
		\caption{An IH-move: a local replacement of a trivalent spine in a surface with boundary.}
		\label{FIG:definition_IH_move}
		\end{figure}

\begin{thm}\label{THM:Luo}
Two trivalent spines of a surface with boundary are related by finitely many IH-moves and isotopies.
\end{thm}

An arc $m$ embedded in a surface $F$ with boundary is {\it proper} if $m \cap \partial F=\partial m$.

\begin{Def}
\label{marking}{\rm
Let $F$ and $M$ be a surface with boundary and a disjoint union of proper arcs in $F$, respectively.
Let $N_{M}$ be a regular neighborhood of $M$ in $F$.
The disjoint union $M$ is a {\it marking} on $F$ if $D$ is a disk, and the disjoint union $D \cap N_M$ of proper arcs in $F$ consists of exactly three arcs for any connected component $D$ of the closure $\operatorname{cl}_F(F\setminus N_M)$, see Fig.~\ref{FIG:definition_marking}.
}\end{Def}
		\begin{figure}[htbp]\centerline{
		\includegraphics{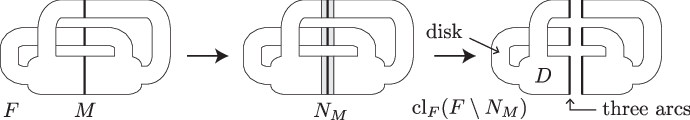}}
		\caption{A marking $M$ on the closure $F$ of a torus minus a disk.}
		\label{FIG:definition_marking}
		\end{figure}

Although a marking $M$ on a surface with boundary is a subset of the surface, we denote by $\mathcal M$ the set of connected components of $M$; then, an element of $\mathcal M$ is a proper arc in $F$.
For a surface $F$ with boundary, there exists a marking on $F$ if and only if $F$ is not a disk, annulus or M\"{o}bius band.
(Other examples of markings are depicted in Fig.~\ref{FIG:example_marking}.)

\begin{rmk}
\label{RMK:marking_IH_move}{\rm
If a marking on a surface $F$ with boundary is given, we can construct exactly one trivalent spine of $F$ up to ambient isotopy.
Conversely, if a trivalent spine of $F$ is given, we can construct exactly one marking for $F$ up to ambient isotopy.
}\end{rmk}

\begin{lemma}\label{LEM:marking_has_same_number_arc}
For markings $L$ and $M$ on a surface $F$ with boundary, it holds that $\#\mathcal L=\#\mathcal M$.
\end{lemma}

\begin{proof}
By the definition of a marking, we have the equality $2(\#\mathcal M)=3|F\setminus M|$ immediately, where $|F\setminus M|$ means the number of connected components of the topological space $F\setminus M$.
We have $\chi(F)=|F\setminus M|-(\#\mathcal M)$ since $F$ is homotopy equivalent to a graph that has $|F\setminus M|$ vertices and $\#\mathcal M$ edges, where $\chi(F)$ is the Euler characteristic of $F$.
Hence, the equality $\#\mathcal M=-3\chi(F)$ holds.
Therefore, the number of arcs in a marking is determined by the Euler characteristic of surfaces.
\end{proof}

Let $M$ be a marking on a surface $F$ with boundary.
For an arc $m\in\mathcal M$, we denote by $D_M(m)$ the connected component of $F\setminus(M\setminus m)$ that contains $m$.
All configurations of $D_M(m)$ are illustrated in Fig.~\ref{FIG:three_cases_configuration}:
the interior $\operatorname{int}D_M(m)$ is an open disk, open annulus, or open M\"{o}bius band.
		\begin{figure}[htbp]\centerline{
		\includegraphics{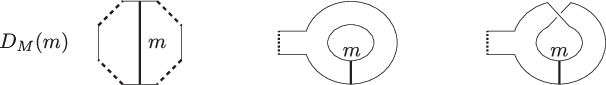}}
		\caption{All probable configurations of $D_M(m)$.}
		\label{FIG:three_cases_configuration}
		\end{figure}
		
An arc $m \in \mathcal M$ is {\it turnable} if $\operatorname{int}D_M(m)$ is an open disk.
For an arc $m \in\mathcal M$ and an arbitrary proper arc $m^*\subset F\setminus(M\setminus m)$,
we write $M(m,m^*)=(M\setminus m)\sqcup m^*$.
Suppose that $m\in\mathcal M$ is turnable.
For a proper arc $m^*\subset F$ as illustrated in Fig.~\ref{FIG:definition_turning}, $M(m,m^*)$ is also a marking on $F$;
$M(m,m^*)$ is said to be obtained from $M$ by {\it turning} $m$ into $m^*$,
and $m$ is {\it turned} into $m^*$.
		\begin{figure}[htbp]\centerline{
		\includegraphics{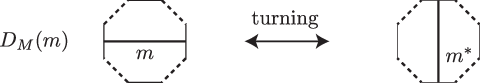}}
		\caption{Turning a turnable arc.}
		\label{FIG:definition_turning}
		\end{figure}
		
\begin{ex}\label{Ex:marking_and_turnable_arcs}{\rm
Let $F$ be a surface with boundary and let $L$ and $M$ be markings on $F$ as illustrated in Fig.~\ref{FIG:example_marking}: $\#\mathcal L=\#\mathcal M=6$.
All arcs in $\mathcal L$ are turnable.
The arc $m\in\mathcal M$ in Fig.~\ref{FIG:example_marking} is not turnable, because $\operatorname{int}D_M(m)$ is an open annulus.
The other arcs in $\mathcal M$ are turnable.
Each trivalent graph in the figure is a trivalent spine of $F$ corresponding to each marking. 
}\end{ex}

		\begin{figure}[htbp]\centerline{
		\includegraphics{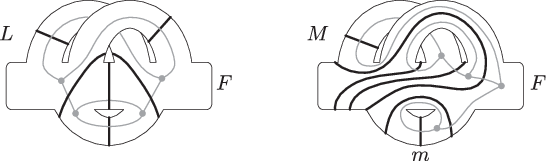}}
		\caption{Markings are colored by black and the corresponding spines are colored by gray.}\label{FIG:example_marking}
		\end{figure}

\begin{rmk}\label{RMK:IH_revolving}{\rm 
Turning an arc corresponds to applying an IH-move of a trivalent spine.
}\end{rmk}
		
		\begin{figure}[htbp]\centerline{
		\includegraphics{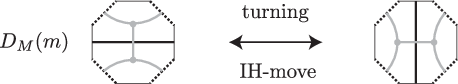}}
		\caption{Turning a turnable arc corresponds to applying an IH-move.}
		\label{FIG:revolving_IH}
		\end{figure}

It can be seen that Theorem~\ref{THM:Luo} follows from Theorem~\ref{THM:marking_rephrase} below by considering Remarks~\ref{RMK:marking_IH_move} and \ref{RMK:IH_revolving}.
We will show Theorem~\ref{THM:marking_rephrase} instead of proving Theorem~\ref{THM:Luo}, in this section.

\begin{thm}\label{THM:marking_rephrase}
If $L$ and $M$ are markings on a surface $F$ with boundary, then there exists a finite sequence $M_0,M_1,\ldots,M_n$ of markings on $F$ such that
\begin{itemize}
\item $M$ (resp.~$L$) and $M_0$ (resp.~$M_n$) are ambient isotopic in $F$, and
\item $M_{i}$ and $M_{i+1}$ are related by turning an arc for any $i$ with $0\le i<n$.
\end{itemize}
\end{thm}

We prepare some notations that are needed for the proof of Theorem~\ref{THM:marking_rephrase}.	
Let $L$ and $M$ be markings on a surface $F$ with boundary such that $L$ intersects $M$ transversally, denoted by $L\transv M$.
We write
\begin{align*}
\mathcal W_L(M)&:=\{\ell\in\mathcal L\mid\ell\cap M\not= \emptyset\},\\
w_L(M)&:=\min\{\#\mathcal W_L(M')\mid\text{$M'$ is ambient isotopic to $M$ in $F$, and $M'\transv L$}\}.
\end{align*}
Markings $L$ and $M$ are in {\it taut position} if there is no disk $\delta$ such that $\delta$ is bounded by a 2-gon consisting of parts of $L$ and $M$ or by a 3-gon consisting of $L$, $M$ and $\partial F$ as illustrated in Fig.~\ref{FIG:taut}.
Let $\ell \in \mathcal L$ be an arc such that $\ell\cap M\not=\emptyset$.
An {\it endarc} of $(\ell;M)$ (resp.~$(L;M)$) is an arc $r$ contained in $\ell$ (resp.~$L$) such that one point of $\partial r$ is in $\partial F$ and the other point is in $M$ and $\operatorname{int}(r)\cap M=\emptyset$.
We note that the number of endarcs of $(L;M)$ is equal to $2\cdot\#\mathcal W_L(M)$.
For an endarc $r$ of $(\ell;M)$, we denote by $m(r;M)$ the arc in $\mathcal M$ satisfying that $m(r;M)$ has an intersection with $r$.
		\begin{figure}[htbp]\centerline{
		\includegraphics{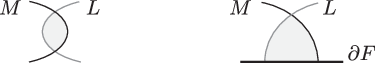}}
		\caption{Two disks bounded by $L$, $M$ and $\partial F$.}
		\label{FIG:taut}
		\end{figure}
\par
We prepare Lemmas~\ref{LEM:taut_then_turnable}, \ref{LEM:zero_crossing_marking} and \ref{LEM:marking_finitely_many_sequence} to show Theorem~\ref{THM:marking_rephrase}.

\begin{lemma}\label{LEM:taut_then_turnable}
Let $L$ and $M$ be markings on a surface $F$ with boundary such that $L$ and $M$ are in taut position.
For any endarc $r$ of $(L;M)$, the arc $m(r;M)\in\mathcal M$ is turnable.
\end{lemma}

\begin{proof}
Suppose that $m:=m(r;M)$ is not turnable.
The configuration of $D_M(m)$ is as illustrated in Fig.~\ref{FIG:proof_switch}, where we remember that $D_M(m)$ is the connected component of $F\setminus(M\setminus m)$ that contains $m$.
One point of $\partial r$ is in $m$, and the other point of $\partial r$ is in the boundary of an annulus or a M\"{o}bius band as illustrated in Fig.~\ref{FIG:proof_switch}.
In each of the cases, a contradiction occurs for our assumption that $L$ and $M$ are in taut position.
Therefore, $m\in\mathcal M$ is turnable.
\end{proof}
		\begin{figure}[htbp]\centerline{
		\includegraphics{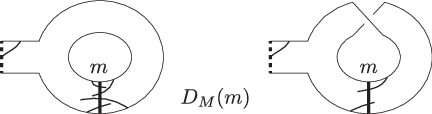}}
		\caption{Two markings in taut position have no nugatory crossings.}
		\label{FIG:proof_switch}
		\end{figure}

Lemmas~\ref{LEM:zero_crossing_marking} and \ref{LEM:marking_finitely_many_sequence} are essential parts for the proof of Theorem~\ref{THM:marking_rephrase}.
		
\begin{lemma}\label{LEM:zero_crossing_marking}
Let $L$ and $M$ be markings on a surface $F$ with boundary satisfying that $w_L(M)=0$.
Then, $L$ and $M$ are ambient isotopic in $F$.
\end{lemma}

\begin{proof}
We deform $M$ by an isotopy of $F$ so that $\#\mathcal W_L(M)=0$, that is, $L\cap M=\emptyset$.
We take a regular neighborhood $N_L$ of $L$ in $F$.
By an isotopy of $F$, we move $M$ so that $M\subset N_L$: such an isotopy exists, since $L$ is a marking.
We show that each connected component of $N_L$, which is a disk, contains exactly one arc of $\mathcal M$ as a subset and the arc is parallel to an arc of $\mathcal L$.
\par
Let $\ell\in\mathcal L$ be an arc, and let $N_\ell$ be the connected component of $N_L$ that contains $\ell$.
Since $M$ is a marking, $\partial F$ and any arc $m\in\mathcal M$  bound no disks in $N_\ell$.
Then, $m$ is parallel to $\ell$.
If $N_\ell$ contains two and more arcs of $\mathcal M$, the arcs are parallel each other.
This leads to a contradiction, since $M$ is a marking.
Then, each connected component of $N_L$ contains no arcs of $\mathcal M$ or contains exactly one arc of $\mathcal M$.
By Lemma~\ref{LEM:marking_has_same_number_arc}, if a connected component of $N_L$ contains no arcs, this leads to a contradiction with $\#\mathcal L=\#\mathcal M$.
Hence, every connected component of $N_L$ contains exactly one arc, and $L$ and $M$ are parallel.
Therefore, $L$ and $M$ are ambient isotopic in $F$.
\end{proof}

\begin{lemma}\label{LEM:marking_finitely_many_sequence}
Let $L$ and $M$ be markings on a surface $F$ with boundary such that $L$ and $M$ are in taut position and $\#\mathcal W_L(M) = w_L(M)>0$.
Then, there is a marking $M'$ such that $M'$ and $M$ are related by finitely many turnings of arcs and $w_L(M')<\#\mathcal W_L(M) $.
\end{lemma}

\begin{proof}In this proof, for any marking $N$ on $F$, we write $\mathcal W(N)=\mathcal W_L(N)$ and $w(N)=w_L(N)$ for short notations.
Let $\ell$ be an arc in $\mathcal L$ such that $\ell\cap M\not=\emptyset$.
\par
{\bf Case 1}: we suppose that $\#(\ell\cap M)=1$.
Let $r$ be an endarc of $(\ell;M)$, and we put $m=m(r;M)$.
Since $L$ and $M$ are in taut position, $m\in\mathcal M$ is turnable by Lemma~\ref{LEM:taut_then_turnable}, that is, $\operatorname{int}D_M(m)$ is an open disk.
We take a proper arc $m^*$ located along $\ell$ as illustrated in Fig.~\ref{FIG:proof_Luo_easy_turning}: $m^*$ is close enough to $\ell$ and $m^*\cap L=\emptyset$, where it may hold that $(L\setminus\ell)\cap M\not=\emptyset$, although $L\setminus\ell$ are not illustrated.
By turning $m$ into $m^*$ in $D_M(m)$, we obtain a new marking $M':=M(m,m^*)$ on $F$.
Now, $L$ and $M$ are in taut position.
Then, any endarc of $(L;M)$ that has no intersections with $m$ is also an endarc of $(L; M')$.
Any endarc $r'$ of $(L;M)$ satisfying that $r'\cap m\not=\emptyset$ and $r'\not\subset\ell$ is contained in an endarc of $(L;M')$.
Both $r$ and the other endarc of $(\ell;M)$ disappear from $M'$.
Since the number of endarcs of $(L; N)$ is equal to $2\cdot\#\mathcal W(N)$ for any marking $N$, we have $\#\mathcal W(M')=\#\mathcal W(M)-1$; then $\#\mathcal W(M')<\#\mathcal W(M)$.
Since $w(M')\le\#\mathcal W(M')$, we have $w(M')<\#\mathcal W(M)$.
		\begin{figure}[htbp]\centerline{
		\includegraphics{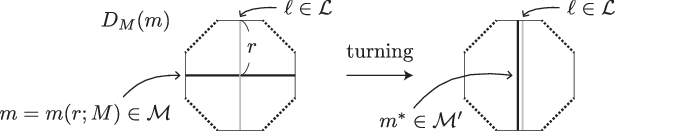}}
		\caption{Configurations of $\ell$, $r$, $m$, $m^*$ and $D_{M}(m)$.}\label{FIG:proof_Luo_easy_turning}
		\end{figure}
\par
{\bf Case 2}: we suppose that $\#(\ell\cap M )>1$.
Set $M_0=M$.
Let $r_0$ be an endarc of $(\ell;M_0)$, and we put $m_0=m(r_0;M_0)$.
Since $L$ and $M_0$ are in taut position, $m_0\in\mathcal M_0$ is turnable, that is $\operatorname{int}D_{M_0}(m_0)$ is an open disk.
Let $r_1$ be the arc contained in $\ell$ such that $r_0\subset r_1$, $\partial(r_1)\subset M_0\cup\partial F$ and $\#(r_1\cap M_0)=2$.
Let $m_1$ the arc in $\mathcal M_0$ such that $\#(m_1\cap r_1) = 1$ and $m_1\not=m_0$, see the left in Fig.~\ref{FIG:proof_Luo_essential_turning}.
\par
We take a proper arc $m_0^*$ located partially along $\ell$ and $m_1$ as illustrated in the right of Fig.~\ref{FIG:proof_Luo_essential_turning}: $m_0^*$ is close enough to a part of $\ell$ and $m_1$, and there is a possibility that $m_0^*$ has intersections with $\ell\setminus r_1$, although $\ell\setminus r_1$ is not illustrated in Fig.~\ref{FIG:proof_Luo_essential_turning}.
By turning $m_0$ into $m_0^*$ in $D_{M_0}(m_0)$, we have a new marking $M_1:=M_0(m_0,m_0^*)$.
Note that $r_1$ is an endarc of $(\ell; M_1)$, $m_1=m(r_1;M_1)$ and $m_1$ is also an arc in $\mathcal M_1$.
		\begin{figure}[htbp]\centerline{
		\includegraphics{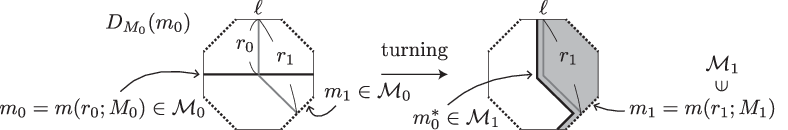}}
		\caption{Configurations of $\ell$, $r_0$, $r_1$, $m_0$, $m_1$, $m_0^*$ and $D_{M_0}(m_0)$.}
		\label{FIG:proof_Luo_essential_turning}
		\end{figure}
\par
We show that (i) $w(M_1)\le\#\mathcal W(M_1)=\#\mathcal W(M_0)$ and (ii) $\#(m_0^\cap\ell)<\#(m_1\cap\ell)$.
Now, $L$ and $M_0$ are in taut position.
Then, by the configuration of $m_0^*$, any endarc $r'$ of $(L;M_0)$ is also an endarc of $(L;M_1)$, $r'$ contains an endarc of $(L;M_1)$, or $r'$ is contained in an endarc of $(L;M_1)$.
Since the number of endarcs of $(L; N)$ is equal to $2\cdot\#\mathcal W(N)$ for any marking $N$, we have $\#\mathcal W(M_1)=\#\mathcal W(M_0)$, then (i) holds.
The arc $r_1$ has an intersection with $m_1$ and has no intersections with $m_0^*$; then we have $\#(m_0^*\cap\ell)\le \#(m_1\cap\ell)-1$, that is, (ii) holds.
\par
If the inequality $w(M_1)<\#\mathcal W(M_0)$ holds, the proof completes.
Then, we suppose that $w(M_1)=\#\mathcal W(M_1)=\#\mathcal W(M_0)$.
Since $L$ and $M_1$ are also in taut position by the configuration of $m_0^*$, the arc $m_1\in\mathcal M_1$ is also turnable by Lemma~\ref{LEM:taut_then_turnable}, that is, $\operatorname{int}D_{M_1}(m_1)$ is an open disk; the shaded region in Fig.~\ref{FIG:proof_Luo_essential_turning} is a half part of $D_{M_1}(m_1)$.
If $\#(m_1\cap\ell) = 1$, the proof completes by considering Case 1.
Then, we suppose that $\#(m_1\cap\ell)>1$.
\par
In the same way to construct $M_1$, we continue turning arcs along a part of $\ell$ while keeping $L$ fixed;
that is, we recursively define $r_{i}$, $m_i$, $m_i^*$ and a marking $M_{i+1}:=M_{i}(m_{i}, m_{i}^*)$ for each integer $i \ge 0$ until it holds that $w(M_n)<\#\mathcal W(M_0)$ or $\#(m_n\cap\ell)=1$ for some integer $n$.
Note that $r_{i}$ is the endarc of $(\ell;M_{i})$ such that $r_{i-1}\subset r_{i}$, and $m_{i}=m(r_{i};M_{i})$, see Fig.~\ref{FIG:proof_Luo_essential_turning_i}.
By the above process, we obtain a sequence $(M_i)_{1\le i\le n}$ of markings satisfying that
\[
\text{ (i$'$) }w(M_i)=\#\mathcal W(M_i)=\#\mathcal W(M_0) \text{ and (ii$'$) }\#\big(m_{i}^*\cap\ell\big)<\#(m_{i+1}\cap\ell)
\]
for any $i$ with $0\le i<n$, in the same manner as (i) and (ii).
\par
It suffices to show that  $(M_i)_{1\le i\le n}$ is finite; then we suppose that $(M_i)_{1\le i\le n}$ is infinite.
We prepare notations only for this proof.
For arcs $\alpha_i\in\mathcal M_i$ and $\alpha_j\in\mathcal M_j$ with $0\le i<j$, we write $\alpha_i\to\alpha_j$ if $\alpha_i=\alpha_j$ as a subset of $F$ or there is an integer $k$ with $i\le k<j$ such that $M_{k}(\alpha_i,\alpha_i^*)=M_{k+1}$ and $\alpha_j=\alpha_i^*$ as a subset of $F$.
For arcs $m\in\mathcal M_0$ and $m'\in\mathcal M_i$, $m'$ is a {\it descendant} of $m$ if there is a finite sequence $\alpha_0$, \ldots, $\alpha_s$ of proper arcs in $F$ such that $\alpha_0=m$, $\alpha_s=m'$ and $\alpha_j\to\alpha_{j+1}$ for any $j$ with $0\le j<s$.
An arc $m \in \mathcal M_0$ is {\it infinite-type} if $m$ continues to be turned endlessly, that is, for any integer $i$, there exists an integer $\kappa$ with $\kappa>i$ such that $M_\kappa(m_\kappa,m_\kappa^*)=M_{\kappa+1}$, where $m_\kappa\in\mathcal M_\kappa$ is the descendant of $m$.
We put
\begin{align*}
\mathcal M_0^{\infty}&=\{m\in\mathcal M_0\mid\text{$m$ is infinite-type}\}, \text{ and}\\
P_i&=\{\#(m'\cap\ell)\mid\text{$m'\in\mathcal M_i$ and $m'$ is an descendant of an infinite-type arc of $\mathcal M_0$}\},
\end{align*}
for any integer $i$, where we define $P_i$ as a multiset, that is, each element of $P_i$ has the multiplicity.
Let $\kappa$ be an integer large enough such that all descendants of $\mathcal M_0\setminus\mathcal M_0^{\infty}$ are no longer turned in $\mathcal M_i$ for any $i$ with $\kappa<i$.
By the inequality (ii$'$), after several times turnings, an arc $m\in\mathcal M_\kappa$ with $\#(m\cap\ell)=\max P_\kappa$ disappears, and a new arc $m'$ with $\#(m'\cap\ell)<\max P_\kappa$ appears instead.
Therefore, there exists an integer $k$ such that $\max P_{\kappa+k}<\max P_\kappa$ or that the multiplicity of $\max P_{\kappa+k}$ is less than that of $\max P_\kappa$, where we remember that $P_i$ is a multiset.
Hence, $\max P_{\kappa'}$ is 0 for an integer $\kappa'$ large enough with $\kappa<\kappa'$.
This leads to a contradiction with the infiniteness of $(M_i)_{1\le i\le n}$; then $(M_i)_{1\le i\le n}$ is finite.
\end{proof}
		\begin{figure}[htbp]\centerline{
		\includegraphics{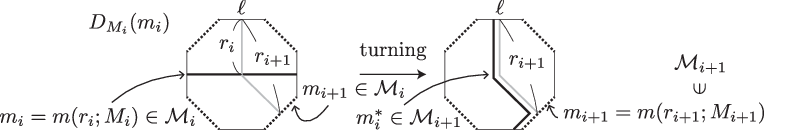}}
		\caption{Configurations of $\ell$, $r_i$, $r_{i+1}$, $m_i$, $m_{i+1}$, $m_i^*$ and $D_{M_i}(m_i)$.}
		\label{FIG:proof_Luo_essential_turning_i}
		\end{figure}

\begin{proof}[Proof of Theorem~\ref{THM:marking_rephrase}]
Let $L$ and $M$ be markings on a surface $F$ with boundary.
If $w_L(M)=0$, the proof completes by Lemma~\ref{LEM:zero_crossing_marking}; then we assume that $w_L(M)>0$.  
By an isotopy of $F$, we deform $M$ so that $L$ and $M$ are in taut position and $\#\mathcal W_L(M)=w_L(M)>0$.
By Lemma~\ref{LEM:marking_finitely_many_sequence}, by finitely many turnings of $M$, we have a marking $M_1$ such that $w_L(M_1)<\#\mathcal W_L(M)$.
By continuing the same operation,
we have a finite sequence $M$, $M_1$, \ldots, $M_n$ of markings such that $w_L(M_n)=0$.
Therefore, $L$ and $M_n$ are ambient isotopic in $F$ by Lemma~\ref{LEM:zero_crossing_marking}.
\end{proof}

\section{A diagram of a spatial surface with boundary}\label{SEC:spatial_surface}
We regard $S^3$ as the one-point compactification $\mathbb{R}^3\cup\{\infty\}$ of the Euclidean space $\mathbb R^3$.
We regard $\mathbb R^2$ as the set $\{(x,y,z)\in\mathbb R^3\mid z=0\}$ and regard $S^2$ as
$\mathbb R^2\cup\{\infty\}$.
Note that $\mathbb R^2\subset\mathbb R^2\cup\{\infty\}=S^2\subset\mathbb R^3\cup\{\infty\}=S^3$.
Let $\mathrm{pr}:\mathbb R^3\to\mathbb R^2;(x,y,z)\mapsto(x,y,0)$ be the canonical projection.
\par
A {\it spatial surface} is a compact surface embedded in $S^3$.
A {\it spatial closed surface} is a spatial surface that is a closed surface.
A {\it spatial surface with boundary} is a spatial surface that is a surface with boundary.
A spatial surface that has $n$ connected components is said to be of {\it $n$-components}.
In Sections~\ref{SEC:spatial_surface} and \ref{SEC:proof}, we call a spatial surface with boundary a spatial surface for short, unless we note otherwise.
If a spatial surface has intersection with $\infty$, we can deform the spatial surface so that it is contained in $\mathbb R^3=S^3\setminus\{\infty\}$ by perturbation around $\{\infty\}$; then we assume that a spatial surface is contained in $\mathbb R^3$.
A disk embedded in $S^3$ is unique up to ambient isotopy; then we assume that a spatial surface has no disk components throughout this paper.
\par
For the same reason as the case of spatial surfaces, we assume that a spatial graph is contained in $\mathbb R^3$.
A spatial graph $G$ is in {\it semiregular position} with respect to $\mathrm{pr}$ if it satisfies that
\begin{itemize}
\item
$\mathrm{pr}(G)$ has finitely many multiple points, that is,
\[
\text{
$\{p\in\mathrm{pr}(G)\mid\#(G\cap\mathrm{pr}^{-1}(p))\ge2\}$ is finite,}
\]
\item
any multiple point $p\in\mathrm{pr}(G)$ is a double point or a triple point, that is,
\[
2\le\#(G\cap\mathrm{pr}^{-1}(p))\le3,
\]
\item
$\mathrm{pr}(v)\ne\mathrm{pr}(u)$ for any distinct trivalent vertices $u$ and $v$ of $G$,
\item
$G\cap\mathrm{pr}^{-1}(p)$ contains no vertices of $G$ for any triple point $p\in\mathrm{pr}(G)$,
\item
there are three disjoint arcs $I$, $I'$ and $I''$ in $G$ such that $\mathrm{pr}(I)$, $\mathrm{pr}(I')$ and $\mathrm{pr}(I'')$ intersect transversally at $p$
for any triple point $p\in\mathrm{pr}(G)$.
\end{itemize}
Let $G$ be a spatial graph in semiregular position.
A multiple point $p\in\mathrm{pr}(G)$ is  {\it regular} if $\#(G\cap\mathrm{pr}^{-1}(p))=2$, $G\cap\mathrm{pr}^{-1}(p)$ contains no vertices of $G$ and there are two disjoint arcs $I$ and $I'$ in $G$ such that $\mathrm{pr}(I)$ and $\mathrm{pr}(I')$ intersect transversally at $p$.
A multiple point $p\in\mathrm{pr}(G)$ is {\it non-regular} if $p$ is not regular.
We note that any multiple point $p\in\mathrm{pr}(G)$ is non-regular if $G\cap\mathrm{pr}^{-1}(p)$ has vertices.
A spatial graph $G$ is in {\it regular position} with respect to $\mathrm{pr}$ if it satisfies that
\begin{itemize}
\item
$G$ is in semiregular position with respect to $\mathrm{pr}$, and
\item
every double point $p\in\mathrm{pr}(G)$ is regular.
\end{itemize}
A {\it projection} of spatial graphs is the image $\mathrm{pr}(G)$ of a spatial graph $G$ in regular position.
A {\it spatial graph diagram} is a pair of a projection of spatial graphs and under/over information at every double point.
For a spatial graph $G$ in regular position, we denote by $D(G)$ the spatial graph diagram satisfying that the projection of $D(G)$ is $\mathrm{pr}(G)$ and that the under/over information at each double point is directly obtained from $G$.

We denote by $V_n(G)$ the set of $n$-valent vertices of a graph $G$.
In this paper, we regard a spatial graph diagram $D$ as a graph; we denote by $V_n(D)$ the set of $n$-valent vertices of $D$.
A {\it $2,3$-graph} is a graph whose any vertex is bivalent or trivalent.
Note that a trivalent graph is a kind of 2,3-graph.

\begin{Def}\label{DEF:sign_of_spatial_graph}{\rm
For a spatial $2,3$-graph diagram $D$, we call a map $s:V_2(D) \to \{\pm 1\}$ a {\it sign} of $D$.
A {\it spatial surface diagram} is a pair $(D,s)$ of a spatial 2,3-graph diagram $D$ and a sign $s$ of $D$.
}\end{Def}

Although any spatial trivalent graph diagram $D$ has no bivalent vertices, we regard the empty function $\emptyset_{\{\pm1\}}:V_2(D)\to\{\pm1\}$ as a sign of $D$.
For a spatial surface diagram $D:=(D,s)$, we construct an unoriented spatial surface $\operatorname{Sf}(D)$ by  operations as illustrated in Fig.~\ref{FIG:construction_of_spatial_surface}:
an arc is replaced with a band, a bivalent vertex is replaced with a twisted band, and one of two crossed bands is slightly perturbed into the direction of the z-axis around each crossing.
An example of a spatial surface diagram $D$ and its spatial surface $\operatorname{Sf}(D)$ is illustrated in Fig.~\ref{FIG:spatial_surface_example}.
		\begin{figure}[htbp]\centerline{
		\includegraphics{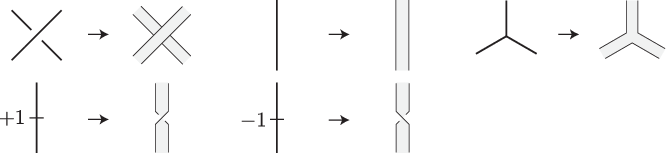}}
		\caption{Operations for spatial surface diagrams.}
		\label{FIG:construction_of_spatial_surface}
		\end{figure}
		
		\begin{figure}[htbp]\centerline{
		\includegraphics{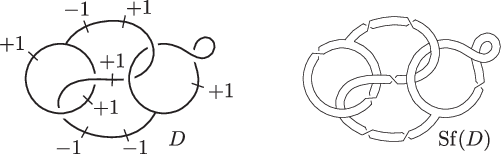}}
		\caption{A spatial surface diagram $D$ and its spatial surface $\operatorname{Sf}(D)$.}
		\label{FIG:spatial_surface_example}
		\end{figure}

\begin{Def}\label{DEF:Reidemeister_move}{\rm
The {\it Reidemeister moves} for a spatial surface diagram are defined as illustrated in Fig.~\ref{FIG:spatial_surface_Reidemeister_move}:
each area outside of a dotted circle is not changed before and after the replacement.
}\end{Def}

Spatial surface diagrams $D$ and $D'$ are said to be {\it related} by R0--R6 moves in $\mathbb R^2$ (resp.~$S^2$)
if there exists a finite sequence $(D_i)_{1\le i\le n}$ of spatial surface diagrams with $D_1=D$ and $D_n=D'$ such that $D_i$ is transformed into $D_{i+1}$ by one of R0--R6 moves in $\mathbb R^2$ (resp.~$S^2$) for any $i$.

\begin{rmk}\label{REM:curl_obtained_from_R0_R1}{\rm
An R${\omega}$ move is realized by using R0 and R1 moves.
}\end{rmk}

As is mentioned in Lemma~\ref{LEM:spatial_surface_has_diagram}~(1), for an arbitrary unoriented spatial surface $F$,
there is a spatial surface diagram $D:=(D,s)$ such that $\operatorname{Sf}(D)$ and $F$ are ambient isotopic in $S^3$,
where we call $D$ a {\it diagram} of $F$, hereafter.
Theorems~\ref{THM:main_non_oriented} and \ref{THM:main_oriented} below are main theorems in this paper.
\begin{thm}\label{THM:main_non_oriented}
Let $D$ and $D'$ be diagrams of unoriented spatial surfaces $F$ and $F'$, respectively.
Then, the following conditions are equivalent:
\begin{enumerate}
\item[(a)] $F$ and $F'$ are ambient isotopic in $S^3$.
\item[(b)] $D$ and  $D'$ are related by R0--R6 moves in $\mathbb{R}^2$.
\end{enumerate}
\end{thm}
		\begin{figure}[htbp]\centerline{
		\includegraphics{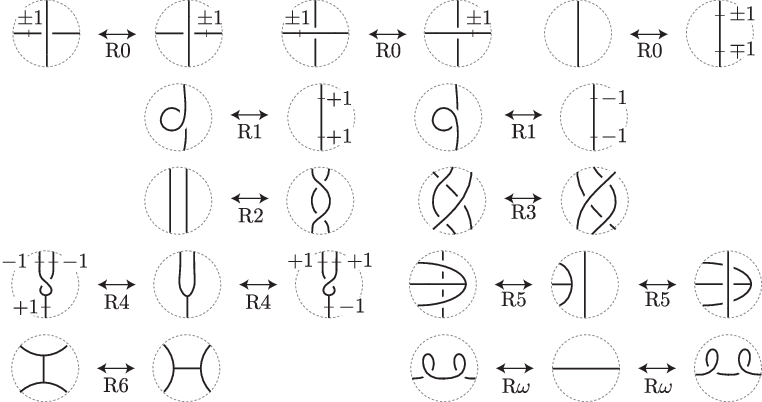}}
		\caption{Reidemeister moves for a spatial surface diagram.}\label{FIG:spatial_surface_Reidemeister_move}
		\end{figure}

Next, we consider oriented spatial surfaces.
Let $D$ be a spatial surface diagram with no bivalent vertices; $\operatorname{Sf}(D)$ is orientable.
We give an orientation for $\operatorname{Sf}(D)$ as illustrated in Fig.~\ref{FIG:oriented_spatial_surface}:
a side of $\operatorname{Sf}(D)$ facing the positive direction of the z-axis is defined as a front side and the other side of $\operatorname{Sf}(D)$ is defined as a back side.
We denote by $\overrightarrow{\operatorname{Sf}}(D)$ the spatial surface $\operatorname{Sf}(D)$ equipped with the above-mentioned orientation.
		\begin{figure}[htbp]\centerline{
		\includegraphics{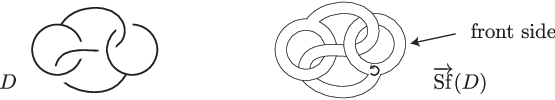}}
		\caption{A spatial surface diagram $D$ with no bivalent vertices, and $\protect\overrightarrow{\operatorname{Sf}}(D)$.}
		\label{FIG:oriented_spatial_surface}
		\end{figure}

As is mentioned in Lemma~\ref{LEM:spatial_surface_has_diagram} (2), for any oriented spatial surface $F$, there is a spatial surface diagram $D$ with no bivalent vertices such that $\overrightarrow{\operatorname{Sf}}(D)$ and $F$ are ambient isotopic in $S^3$ including orientations, where we call $D$ a {\it diagram} of $F$, hereafter.

\begin{thm}\label{THM:main_oriented}
Let $D$ and $D'$ be diagrams of oriented spatial surfaces $F$ and $F'$, respectively.
Then, the following conditions are equivalent.
\begin{enumerate}
\item[(a)]
$F$ and $F'$ are ambient isotopic in $S^3$ including orientations.
\item[(b)]
$D$ and $D'$ are related by R${\omega}$, R2--R3 and R5--R6 moves in $\mathbb R^2$.
\item[(c)]
$D$ and $D'$ are related by R2--R3 and R5--R6 moves in $S^2$.
\end{enumerate}
\end{thm}

The proofs of Theorems~\ref{THM:main_non_oriented} and \ref{THM:main_oriented} are written in Section~\ref{SEC:proof}.
When we deal with a spatial surface that contains closed components, we consider a spatial surface with boundary that is obtained by removing an open disk in each closed component.
Details are described in Section~\ref{SEC:non_split_spatial_surface}.

\begin{rmk}{\rm
In~\cite{ishii2008}, Ishii established Reidemeister moves of handlebody-links.
The Reidemeister moves of handlebody-link is the R1--R6 moves in Fig.~\ref{FIG:spatial_surface_Reidemeister_move} by omitting all signs.
We obtain the Reidemeister moves of oriented spatial surfaces by removing R1 and R4 moves of that of handlebody-links. 
Then, we also call an oriented spatial surface a \textit{framed handlebody-link}.
}\end{rmk}

\section{The proofs of Theorems~\ref{THM:main_non_oriented} and \ref{THM:main_oriented}}\label{SEC:proof}
Throughout Section~\ref{SEC:proof}, a bivalent vertex $v$ of a spatial surface diagram $D:=(D,s)$ is denoted by a black (resp.~white) vertex if $s(v)=+1$ (resp.~$s(v)=-1$), see Fig.~\ref{FIG:other_twist_band}.
A spatial surface $\operatorname{Sf}(D)$ is already defined in Definition~\ref{DEF:sign_of_spatial_graph}.
However, in Section~\ref{SEC:proof}, we replace each bivalent vertex with the band as illustrated in Fig.~\ref{FIG:other_twist_band} to avoid complication of figures: twisted bands in Fig.~\ref{FIG:other_twist_band} and Fig.~\ref{FIG:construction_of_spatial_surface} are ambient isotopic in a local area.
		\begin{figure}[htbp]\centerline{
		\includegraphics{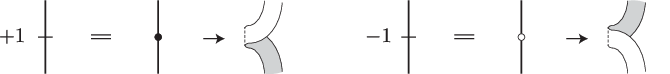}}
		\caption{A black/white vertex and a twisted band.}
		\label{FIG:other_twist_band}
		\end{figure}
\par
We prepare some notations used in proofs.
Let $D$ and $G$ be a spatial surface diagram and a trivalent spine of $F:=\operatorname{Sf}(D)$, respectively.
A proper arc $\alpha$ in $F$ is a {\it corner-arc} if the restriction $\mathrm{pr}|_{N_p}$ is not injective for any neighborhood $N_p$ of any point $p\in\alpha$, see Fig.~\ref{FIG:corner_arc_definition}.
We denote by $A_F$ the disjoint union of corner-arcs of $F$.
The pair $(G;F)$ is in {\it semiregular position} if
\begin{itemize}
\item $G$, which is a spatial graph, is in semiregular position (see Section~\ref{SEC:spatial_surface}), and
\item $G \cap A_F$ is finite, that is, $G$ has intersection with $A_F$ at finitely many points.
\end{itemize}
		\begin{figure}[htbp]\centerline{
		\includegraphics{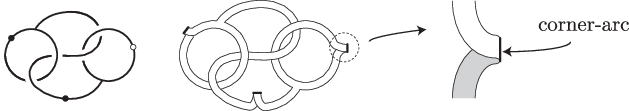}}
		\caption{Thick proper arcs mean corner-arcs of $\operatorname{Sf}(D)$.}
		\label{FIG:corner_arc_definition}
		\end{figure}
\par
When $(G;F)$ is in semiregular position, $\mathrm{pr}(G)$ has no triple points, since $G\subset F$ and $\#(F\cap\mathrm{pr}^{-1}(p))\le 2$ for any $p\in\mathbb R^2$.
Suppose that $(G;F)$ is in semiregular position.
A point $p\in A_F\cap G$ is {\it transversal} if $p$ is not a trivalent vertex of $G$ and $A_F$ and $G$ intersect transversally at $p$.
A point $p\in A_F\cap G$ is {\it non-transversal} if $p$ is not transversal.
A trivalent vertex is non-transversal if it is in $A_F$.
The pair $(G;F)$ is in {\it regular position} if
\begin{itemize}
\item $(G;F)$ is in semiregular position,
\item $G$ is in regular position (see Section~\ref{SEC:spatial_surface}), and
\item any point in $A_F\cap G$ is transversal.
\end{itemize}
\par
Suppose that $(G;F)$ is in regular position.
We explain a way to construct a spatial surface diagram $D(G;F)$. 
First, we replace every point in $G\cap A_F$ with a bivalent vertex.
We regard $G$ as the spatial 2,3-graph obtained by adding bivalent vertices.
Next, we define a sign $s(G;F)$ of $D(G)$.
We fix a ``thin''  regular neighborhood $N_G$ of $G$ in $F$;
$N_G$ and $F$ are ambient isotopic in $\mathbb R^3 \subset S^3$.
For every bivalent vertex $v$ in $D(G)$, which is on a corner-arc of $F$, we define $s(G;F)(v)$ by observing a peripheral area around $v$ as depicted in Fig.~\ref{FIG:induced_sign_definition}.
Then, we have the spatial surface diagram $(D(G), s(G;F))$, denoted by $D(G;F)$ hereafter.
		\begin{figure}[htbp]\centerline{
		\includegraphics{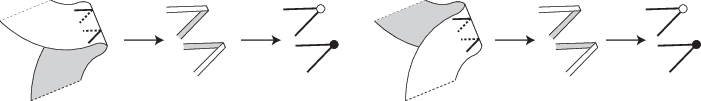}}
		\caption{Procedure for obtaining the sign $s(G;F)$ of $D(G)$.}
		\label{FIG:induced_sign_definition}
		\end{figure}
\par
We also write $\operatorname{Sf}(G;F)=\operatorname{Sf}\bigl(D(G;F)\bigr)$.
We often assume that  $\operatorname{Sf}(G;F)$ is in $F$, that is, we identify $\operatorname{Sf}(G;F)$ with $N_G$.
\par
The lemma below claims that any spatial surface can be presented by spatial surface diagrams.

\begin{lemma}\label{LEM:spatial_surface_has_diagram}
Let $F$ be an spatial surface.
\begin{enumerate}
\item[(1)]
There is a spatial surface diagram $D$ such that $\operatorname{Sf}(D)$ and $F$ are ambient isotopic in $S^3$.
\item[(2)]
If $F$ is orientable and oriented, then there is a spatial surface diagram $D$ with no bivalent vertices such that $\overrightarrow{\operatorname{Sf}}(D)$ and $F$ are ambient isotopic in $S^3$ including orientations.
\end{enumerate}
\end{lemma}

\begin{proof}
We describe an example of ways to construct a diagram of $F$.
\par
We prove (1).
By perturbing $F$ slightly, we deform $F$ so that $U\cap C\ne\emptyset$ and $\mathrm{pr}|_{U}$ is injective for some open subset $U$ of $F$ and each connected component $C$ of $F$.
We fix an arbitrary trivalent spine $G$ of $F$.
By an isotopy of $F$, we deform $G$ so that every trivalent vertex of $G$ is contained in $U$.
We perturb the z-axis of the canonical projection $\mathrm{pr}$ slightly so that $G$ is in regular position.
Next, we take a thin enough neighborhood $N_G$ of $G$ in $F$.
If necessary, we deform any twisted part of $N_G$ into the configuration in Fig.~\ref{FIG:other_twist_band} so that $(G;N_G)$ is in regular position.
Since $N_G$, $F$ and $\operatorname{Sf}(G;N_G)$ are ambient isotopic, $D(G;N_G)$ is a diagram of $F$.
\par
We prove (2).
By the consequence of (1), there is a spatial surface diagram $D:=(D,s)$ such that $\operatorname{Sf}(D)$ and $F$ are ambient isotopic, where we forget the orientation of $F$.
We regard $F$ as $\operatorname{Sf}(D)$.
Let $G$ be a 2,3-spine of $F$ such that $(D(G),s(G;F))=D$.
By an isotopy of $S^3$, we deform the oriented surface $F$ so that any front side of peripheral regions around trivalent vertices of $G$ faces the positive direction of the z-axis.
By an isotopy of $S^3$ that keeps any peripheral region of trivalent vertices fixed, we deform $F$ so that $(G;F)$ is in regular position and $F$ has no corner-arcs:
all full twisted bands are deformed as illustrated in Fig.~\ref{FIG:full_twist_to_curl}.
Then, we obtain a spatial surface diagram $D'=D(G;F)$ with no bivalent vertices such that $\overrightarrow{\operatorname{Sf}}(D')$ and $F$ are ambient isotopic in $S^3$ including orientations.
\end{proof}
		\begin{figure}[htbp]\centerline{
		\includegraphics{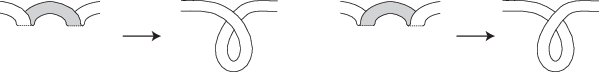}}
		\caption{A full twist band is deformed into a teardrop-like band.}
		\label{FIG:full_twist_to_curl}
		\end{figure}

\subsection{The proof of Theorem~\ref{THM:main_non_oriented}}\label{SEC:non_oriented}
For also spatial trivalent graph diagrams, Reidemeister moves are defined as illustrated in Fig.~\ref{FIG:spatial_graph_Reidemeister}
(cf.~\cite{kauffman1989}).
Two spatial trivalent graph diagrams are related by Reidemeister moves if and only if the two spatial graphs are ambient isotopic in $S^3$.
We use the same symbols as the spatial surface diagrams; an R2, R3 or R5 move of Fig.~\ref{FIG:spatial_graph_Reidemeister} is the same as that of Fig.~\ref{FIG:spatial_surface_Reidemeister_move}.
\par
For a spatial surface diagram $D$, let us  denote by $\overline{D}$ the spatial trivalent graph diagram that is obtained from $D$ by forgetting its bivalent vertices.
A spatial surface $F$ is {\it well-formed} if $F=\operatorname{Sf}(D)$ for some spatial surface diagram $D$, where this notation is used only in Lemma~\ref{LEM:step1(simple_spine_into_complicated_spine)}.
		\begin{figure}[htbp]\centerline{
		\includegraphics{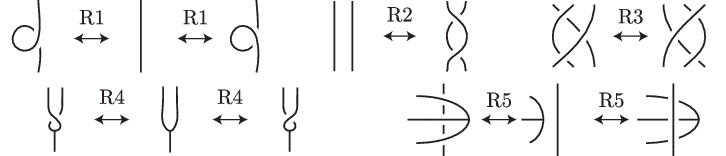}}
		\caption{Reidemeister moves for spatial trivalent graph diagrams.}
		\label{FIG:spatial_graph_Reidemeister}
		\end{figure}
\par
Let $\{h_t:\mathbb R^3\to\mathbb R^3\}_{t\in[0,1]}$ be an isotopy with $h_0=\mathrm{id}|_{\mathbb R^3}$, and let $G$ be a spatial graph.
For $c\in[0,1]$, the image $\mathrm{pr}(h_c(G))$ has a \textit{critical} point $p\in\mathrm{pr}(h_c(G))$ if there is a small open interval $R$ with $c\in R\subset[0,1]$ such that
\begin{itemize}
\item $h_r(G)$ is in regular position with respect to $\mathrm{pr}$ for any $r\in R$, and
\item a transversal crossing of projections of $G$ occurs or disappears around $p$ before and after $c$ as illustrated in Fig.~\ref{FIG:critical}.
\end{itemize}
		\begin{figure}[htbp]\centerline{
		\includegraphics{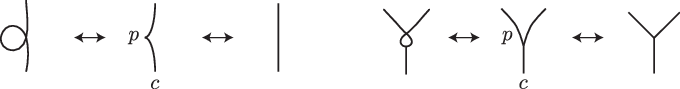}}
		\caption{A critical point $p\in\mathrm{pr}(h_c(G))$.}
		\label{FIG:critical}
		\end{figure}

\begin{lemma}
\label{LEM:step1(simple_spine_into_complicated_spine)}
Let $D$ be a spatial surface diagram, and we put $F=\operatorname{Sf}(D)$.
We denote by $G$ the trivalent spine of $F$ satisfying $D(G;F)=D$.
Let $\{h_t:\mathbb R^3\to\mathbb R^3\}_{t\in[0,1]}$ be an isotopy with $h_0=\mathrm{id}|_{\mathbb R^3}$ such that
$h_1(F)$ is well-formed and $(h_1(G);h_1(F))$ is in regular position.
Then, $D$ and $D(h_1(G); h_1(F))$ are related by R0--R5 moves in $\mathbb R^2$.
\end{lemma}

\begin{proof}
For any subset $X$ in $\mathbb R^3$ and any $t\in [0,1]$, we write $X_t=h_t(X)$.
We slightly perturb the isotopy or the z-axis of the canonical projection $\mathrm{pr}$ so that the following conditions are satisfied:
\begin{itemize}
\item $G_t$ is in semiregular position for any $t\in[0,1]$,
\item $S:=\{s\in[0,1]\mid\text{$G_s$ is not in regular position}\}$ is finite,
\item $C:=\{c\in[0,1]\mid\text{$\mathrm{pr}(G_c)$ has a critical point}\}$ is finite and $C\subset[0,1]\setminus S$,
\item if $s\in S$, then $\mathrm{pr}(G_s)$ has exactly one non-regular multiple point,
\item if $c\in C$, then $\mathrm{pr}(G_c)$ has exactly one critical point.
\end{itemize}
We note that $G_r$ is in regular position for any $r\in R:=[0,1]\setminus S$.
For any $r \in R$, we write $E_r=D(G_r)$, which is a spatial trivalent graph diagram.
Note that $E_0=\overline{D}$ and $E_1=\overline{D(h_1(G); h_1(F))}$.
Let $t_1$, $t_2$, \ldots, $t_{n}$ be all points in $S\sqcup C$ such that $t_1<\cdots<t_n$, where we put $n=\#(S\sqcup C)$.
We put $R_i=\{r\in[0,1]\mid t_i<r<t_{i+1}\}$ for any integer $i$ with $0\le i\le n$, where we set $t_0=0$ and $t_{n+1}=1$.
\par
By adding bivalent vertices and signs to the sequence $(E_r)_{r\in R\setminus C}$ of spatial trivalent graph diagrams, we will construct a sequence $(D_r)_{r\in R\setminus C}$ of spatial surface diagrams with $D=D_0$ such that $D_{r_i}$ and $D_{r_{i+1}}$ are related by R0--R5 moves for any $i$ with $1\le i <n$ and any $r_{i}\in R_{i}$ and $r_{i+1}\in R_{i+1}$.
Furthermore, we will show that $D_1$ and $D(h_1(G);h_1(F))$ are related by R0 moves; then the proof will complete.
\par
First, we recursively construct a sequence $(D_r)_{r\in R\setminus C}$ of spatial surface diagrams satisfying $\overline{D_r}=E_r$ for any $r\in R\setminus C$.
We set $D_0=D$.
We define $D_{r_0}$ for any $r_0\in R_0$ as follows: $D_{r_0}$ and $D_0$ are related by an isotopy of $\mathbb R^2$ and $\overline{D_{r_0}}=E_{r_0}$.
Suppose that $(D_r)_{r\in R_0\cup\cdots\cup R_i}$ is already defined.
We define $(D_r)_{r\in R_0\cup\cdots\cup R_{i+1}}$ by the following procedures.
\par
{\bf Case 1:}
we consider the case that $t_{i+1}\in C$, that is, an R1 (resp.~R4) move of Fig.~\ref{FIG:spatial_graph_Reidemeister} is applied before and after $t_{i+1}$ in $(E_r)_{r\in R\setminus C}$.
Let $\delta$ be a small region in which the move is applied.
If a bivalent vertex of the corresponding spatial surface diagram is contained in $\delta$, we move the vertex on the outside of $\delta$ by using R0 moves.
By adding bivalent vertices, we define $D_{r_{i+1}}$ for any ${r_{i+1}}\in R_{i+1}$ as illustrated in Fig.~\ref{FIG:proof_replace_01}:
R0 and R1 (resp.~R4) moves of Fig.~\ref{FIG:spatial_surface_Reidemeister_move} are applied before and after $t_{i+1}$, $\overline{D_{r_{i+1}}}=E_{r_{i+1}}$, and $D_{r_{i+1}}$ and $D_{r_{i+1}'}$ are related by R0 moves and an isotopy of $\mathbb R^2$ for any $r_{i+1}'\in R_{i+1}$.
The case of applying the right R1 (resp.~R4) move of Fig.~\ref{FIG:spatial_graph_Reidemeister} is omitted.
		\begin{figure}[htbp]\centerline{
		\includegraphics{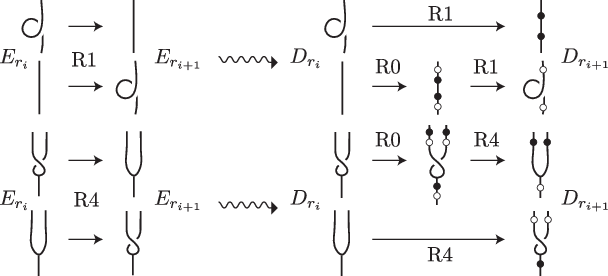}}
		\caption{Applying R0 and R1 (resp.~R4) moves of spatial surface diagrams.}
		\label{FIG:proof_replace_01}
		\end{figure}
\par
{\bf Case 2:}
we consider the case that $t_{i+1}\in S$, that is, a move except for both R1 and R4 moves of Fig.~\ref{FIG:spatial_graph_Reidemeister} is applied before and after $t_{i+1}$ in $(E_r)_{t\in R\setminus C}$.
If a bivalent vertex of the corresponding spatial surface diagram is contained in $\delta$, we move it on the outside of $\delta$ by R0 moves.
By adding no bivalent vertices, we define $D_{r_{i+1}}$ for any ${r_{i+1}}\in R_{i+1}$ as follows:
$D_{r_{i+1}}$ is equal to $E_{r_{i+1}}$ in $\delta$, $\overline{D_{r_{i+1}}}=E_{r_{i+1}}$, and $D_{r_{i+1}}$ and $D_{r_{i+1}'}$ are related by R0 moves and an isotopy of $\mathbb R^2$ for any $r_{i+1}'\in R_{i+1}$.
Before and after $t_{i+1}$, a move that is not illustrated in Fig.~\ref{FIG:spatial_graph_Reidemeister} may be applied, however, $D_{r_{i}}$ and $D_{r_{i+1}}$ are actually related by R2, R3 and R5 moves;
two examples of a move that is not illustrated in Fig.~\ref{FIG:spatial_graph_Reidemeister} are depicted in Fig.~\ref{FIG:proof_other}.
		\begin{figure}[htbp]\centerline{
		\includegraphics{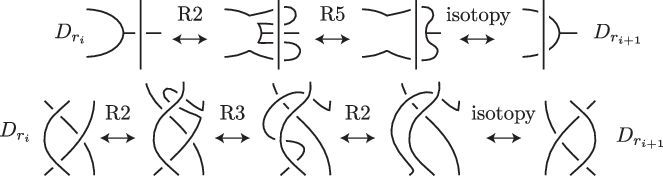}}
		\caption{Deformations by using R2, R3 and R5.}
		\label{FIG:proof_other}
		\end{figure}
\par
By the processes, we have $(D_r)_{r\in R\setminus C}$ of spatial surface diagrams satisfying that $\overline{D_r}=E_r$ for any $r\in R\setminus C$; then $D$ and $D_1$ are related by R0--R5 moves.
\par
Secondly, we show that $D_1$ and $D(h_1(G);h_1(F))$ are related by R0 moves, where we note that $D(G_1;F_1)=D(h_1(G);h_1(F))$.
We take a thin enough regular neighborhood $V$ of $G$ in $\mathbb R^3$ so that $F_t':=V_t\cap F_t$ is a regular neighborhood of $G_t$ in $F_t$ for any $t\in[0,1]$.
Now, $V_t$ is a disjoint union of handlebodies embedded in $\mathbb R^3$, and $F_t'$ is a spatial surface properly embedded in $V_t$, that is, $\partial F_t'\subset\partial V_t$ and $\operatorname{int}F_t'\subset\operatorname{int}V_t$.
\par
Although each $F_t'$ is properly embedded in $V_t$, we regard that each $\operatorname{Sf}(D_t)$ is also properly embedded in $V_{t}$ and that $\operatorname{Sf}(D_t)$ contains $G_t$ for any $t\in I$, where we note that $G_t$ is also a spine of $\operatorname{Sf}(D_t)$.
We regard that $F_1'=\operatorname{Sf}(G_1;F_1)$.
By the construction of $(\operatorname{Sf}(D_t))_{t\in I}$, the spatial surfaces $\operatorname{Sf}(D_1)$ and $F_1'$ are related by an isotopy $\{\varphi_t: V_1\to V_1\}_{t\in[0,1]}$ satisfying that $\varphi_0=\mathrm{id}|_{V_1}$, $\varphi_1(\operatorname{Sf}(D_1))=F_1'$ and $\varphi_t(x)=x$ for any $t \in I$ and any $x\in G_1$.
Let $e$ be an edge of $G_1$.
Since $\{\varphi_t\}_{t\in[0,1]}$ keeps $G_1$ fixed, then there exist bands $B:=e\times J\subset\operatorname{Sf}(D_1)$ and $B':=e\times J'\subset F_1'$ such that $\varphi_1(B)=B'$ and $e\subset\partial e\times J=\partial e\times J'\subset G_1$, where each of $J$ and $J'$ is a closed interval.
An example of $B$ is illustrated in Fig.~\ref{FIG:proof_thin_band}.
		\begin{figure}[htbp]\centerline{
		\includegraphics{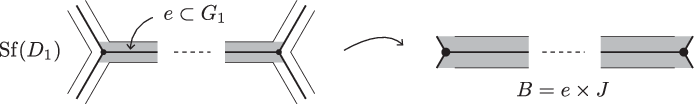}}
		\caption{A band $B=e\times J$, where $J$ is a segment homeomorphic to $[0,1]$.}\label{FIG:proof_thin_band}
		\end{figure}
\par
The sum of the sign of vertices on $\mathrm{pr}(e)$ in $D_1$ is equal to that in $D(G_1;F_1)$.
Therefore, $D_1$ and $D(G_1;F_1)$ are related by R0 moves.
\end{proof}

\begin{lemma}\label{LEM:Step2(Luo_spatial_surface_version)}
Let $D$ be a spatial surface diagram, and put $F=\operatorname{Sf}(D)$.
Let $G$ be a trivalent spine of $F$ such that $(G;F)$ is in regular position.
Then,
\begin{enumerate}
\item[(1)] $D(G;F)$ and $D$ are related by R0--R2 and R4--R6 moves in $\mathbb R^2$, and
\item[(2)] $D(G;F)$ and $D$ are related by R2 and R5--R6 moves in $\mathbb R^2$, if $D$ has no bivalent vertices.
\end{enumerate}
\end{lemma}

\begin{proof}
We show (1).
Let $G_0$ be the spine of $F$ satisfying $D(G_0;F)=D$.
Since two trivalent spines are related by finitely many IH-moves and isotopy (Theorem~\ref{THM:Luo}), it suffices to show that $D$ and $D(G;F)$ are related by R0--R2 and R4--R6 moves moves if $G_0$ and $G$ are related by exactly one IH-move or an isotopy in $F$.
\par
First, we suppose that $G_0$ and $G$ are related by an isotopy in $F$; we show that $D$ and $D(G;F)$ are related by R0--R2 and R4--R5 moves moves.
Let $\{h_t:F\to F\}_{t\in[0,1]}$ be an isotopy of $F$ with $h_0=\mathrm{id}|_{F}$ such that $h_1(G_0)=G$.
Write $G_t=h_t(G_0)$ for any $t\in[0,1]$.
We slightly perturb the isotopy or the z-axis of $\mathrm{pr}$ so that the following conditions are satisfied:
\begin{itemize}
\item $(G_t;F)$ is in semiregular position for any $t\in[0,1]$,
\item $S:=\{s\in[0,1]\mid\text{$(G_s;F)$ is not in regular position}\}$ is finite,
\item $C:=\{c\in[0,1]\mid\text{$\mathrm{pr}(G_c)$ has a critical point}\}$ is finite and $C\subset[0,1]\setminus S$,
\item if $s\in S$, then
	\begin{itemize}
	\item $G_s$ has exactly one non-transversal intersection with $A_F$ and any double point of $\mathrm{pr}(G_s)$ is regular, otherwise
	\item $\mathrm{pr}(G_s)$ has exactly one non-regular double point and any intersection of $G_s$ with $A_F$ is transversal,
	\end{itemize}
\item if $c\in C$, then $\mathrm{pr}(G_c)$ has exactly one critical point,
\end{itemize}
Put $R=[0,1]\setminus S$.
For any $r\in R$, $(G_r;F)$ is in regular position;
$D(G_r; F)$ is well-defined.
We check the sequence $(D(G_r; F))_{r\in R}$ of spatial surface diagrams in detail.
Let $t$ be an point in $S\sqcup C$.
\par
{\bf Case 1}: we consider the case where $t\in S$ and $G_{t}$ has one non-transversal intersection $v$ with $A_F$.
We note that any double point of $\mathrm{pr}(G_t)$ is regular.
Suppose that $v$ is not a trivalent vertex.
Before and after $t$ in $(D(G_r; F))_{r\in R}$, one of the deformations in Fig.~\ref{FIG:occur_R1} is applied, although the case where an arc passes through a corner-arc from the back side of $F$ is omitted.
Each deformation is realized by R0--R1 moves.
We suppose that $v$ is a trivalent vertex.
Before and after $t$ in $(D(G_r;F))_{r\in R}$, one of the deformations in Fig.~\ref{FIG:occur_R4} is applied, although the case where a vertex passes through a corner-arc from the back side of $F$ is omitted.
Each deformation is realized by R0 and R4 moves.
\par
{\bf Case 2}: we consider the case where $t\in S$ and $\mathrm{pr}(G_{t})$ has one non-regular double point.
We note that any intersection of $G_t$ with $A_F$ is transversal.
Since no arcs pass through a corner-arc before and after $t$, then R0 moves are not applied in $(D(G_r;F))_{r\in R}$ before and after $t$.
A deformation before and after $t$ in $(D(G_r;F))_{r\in R}$  are realized by R2 and R5 moves.
		\begin{figure}[htbp]\centerline{
		\includegraphics{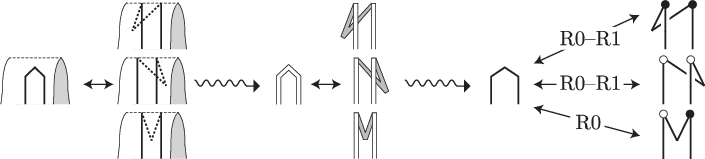}}
		\caption{Processes that an arc passes through a corner-arc from the front side.}
		\label{FIG:occur_R1}
		\end{figure}

		\begin{figure}[htbp]\centerline{
		\includegraphics{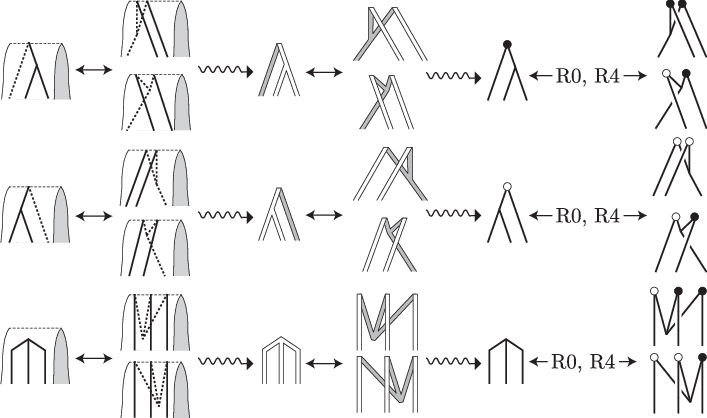}}
		\caption{Processes that a trivalent vertex passes through a corner-arc from the front side.}\label{FIG:occur_R4}
		\end{figure}
\par
{\bf Case 3}: we consider the case where $t\in C$.
We note that $G_t$ is in regular position.
There is a transversal point $p\in G_t\cap A_F$ such that $\mathrm{pr}(p)\in\mathrm{pr}(G_t)$ is the critical point.
One of the deformations in Fig.~\ref{FIG:occur_R0_R1} is applied in $(D(G_r; F))_{r\in R}$ before and after $t$.
Each deformation is realized by R0--R1 moves.
\par
By considering the cases above, $D$ and $D(G;F)$ are related by R0--R2 and R4--R5 moves, where we note $D=D(G_0;F)$ and $D(G;F)=D(G_1;F)$.
\par
Secondly, we suppose that $G_0$ and $G$ are related by an IH-move on $F$; we show that $D$ and $D(G;F)$ are related by R0--R2 and R4--R6 moves.
When an IH-move is applied, there might be many arcs above or below the region where the IH-move is applied.
Then, we shrink the region, by isotopy of $F$, into a small region so that an R6 move can be applied. 
In the process of shrinking the region, R0--R2 and R4--R5 moves are applied, see the cases 1, 2 and 3.
Hence, $D$ and $D(G;F)$ are related by R0--R2 and R4--R6 moves.
		\begin{figure}[htbp]\centerline{
		\includegraphics{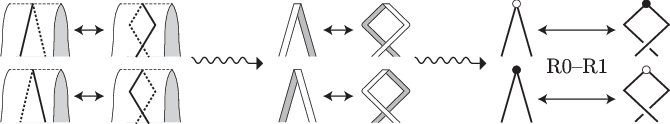}}
		\caption{All deformations around a corner-arc.}
		\label{FIG:occur_R0_R1}
		\end{figure}

We show (2).
Since $F$ has no corner-arcs, none of R0--R1 and R4 moves are not applied, see the case 1, 2 and 3.
Therefore, $D$ and $D(G;F)$ are related by R2 and R5--R6 moves.
\end{proof}

\begin{proof}[Proof of Theorem~\ref{THM:main_non_oriented}.] 
Since $F$ and $\operatorname{Sf}(D)$ are ambient isotopic, we assume that $\operatorname{Sf}(D)=F$.
For the same reason, we assume that $\operatorname{Sf}(D')=F'$.
\par
Suppose that (b) is satisfied.
If $D$ and $D'$ are related by exactly one of R0--R6 moves, we can immediately construct an isotopy between $F$ and $F'$.
Since $D$ and $D'$ are related by finitely many R0--R6 moves, then $F$ and $F'$ are also ambient isotopic; (a) holds.
\par
Suppose that (a) is satisfied; we show that (b) holds.
Let $G$ (resp.~$G'$) be the trivalent spine of $F$ (resp.~$F'$) such that $D(G;F)=D$ (resp.~$D(G';F')=D'$).
Since $F$ and $F'$ are ambient isotopic, we take an isotopy $\{h_t:\mathbb R^3\to\mathbb R^3\}_{t\in [0,1]}$ with $h_0=\mathrm{id}|_{\mathbb R^3}$ such that $h_1(F)=F'$ and $(h_1(G);h_1(F))$ is in regular position.
By Lemma~\ref{LEM:step1(simple_spine_into_complicated_spine)}, $D$ and $D(h_1(G);h_1(F))$ are related by R0--R5 moves.
Furthermore, $D(h_1(G);h_1(F))$ and $D'$ are related by R0--R2 and R4--R6 moves by Lemma~\ref{LEM:Step2(Luo_spatial_surface_version)}~(1).
Therefore, $D$ and $D'$ are related by R0--R6 moves.
\end{proof}

\subsection{The proof of Theorem~\ref{THM:main_oriented}.}\label{SEC:oriented}
\begin{lemma}\label{LEM:easy_replacement}
\begin{enumerate}
\item[(1)] The local replacement in Fig.~\ref{FIG:R2_R3_curl} is realized by R2--R3 moves in the shaded region of the figure.
\item[(2)] If a spatial surface diagram $D$ has no bivalent vertices, then an R${\omega}$ move as illustrated in Fig.~\ref{FIG:spatial_surface_Reidemeister_move} is realized by R2--R3 and R5 moves on $S^2$.
\end{enumerate}
\end{lemma}

\begin{proof}
We show (1).
The left local replacement in Fig.~\ref{FIG:R2_R3_curl} is realized by R2--R3 moves on $\mathbb R^2$ as depicted in Fig.~\ref{FIG:proof_R2_R3_curl}.
The right replacement in Fig.~\ref{FIG:R2_R3_curl} is also realized by R2--R3 moves in the same manner.
\par
We show (2).
When the thick part of the arc in Fig.~\ref{FIG:proof_curl} is passing through the backside of the sphere $S^2$, R2--R3 and R5 moves are applied.
The right R${\omega}$ move in Fig.~\ref{FIG:spatial_surface_Reidemeister_move} is also realized by the same moves.
\end{proof}
		\begin{figure}[htbp]\centerline{
		\includegraphics{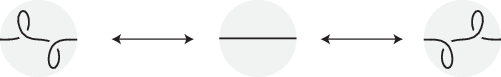}}
		\caption{A local replacement.}
		\label{FIG:R2_R3_curl}
		\end{figure}
		
		\begin{figure}[htbp]\centerline{
		\includegraphics{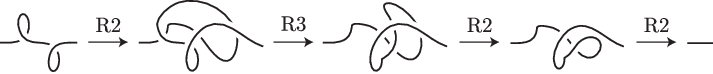}}
		\caption{This deformation is realized  by R2--R3 moves in $\mathbb R^2$.}\label{FIG:proof_R2_R3_curl}
		\end{figure}		

		\begin{figure}[htbp]\centerline{
		\includegraphics{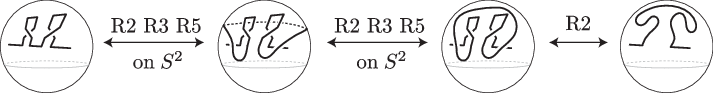}}
		\caption{The R${\omega}$ move is realized by R2--R3 and R5 moves in $S^2$.}
		\label{FIG:proof_curl}
		\end{figure}
\par
For a spatial surface diagram $D$, let us denote by $\overline{D}$ the spatial surface diagram that is obtained from $D$ by forgetting its bivalent vertices:
$\overline{D}$ has no bivalent vertices, and the sign of $\overline{D}$ is the empty function.
An {\it edge} of the spatial surface diagram $D$ is a part $P$ of $D$ such that $P$ corresponds with the projection image of an edge or a loop component of the trivalent graph which $\overline{D}$ represents. 
A self-crossing of $D$ is {\it positive} (resp.~{\it negative}) if it is the left (resp.~right) self-crossing of Fig.~\ref{FIG:positive_negative}: a dashed line may have crossings with other arcs of $D$.
		\begin{figure}[htbp]\centerline{
		\includegraphics{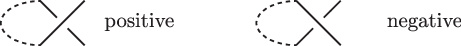}}
		\caption{A positive self-crossing and a negative self-crossing}
		\label{FIG:positive_negative}
		\end{figure}

\begin{Def}\label{DEF:framing}{\rm
For an edge $e$ of a spatial surface diagram $(D,s)$, let us denote by $p(e)$ and $n(e)$ the number of positive and negative self-crossings of $e$, respectively.
We define the {\it framing} $f(e)$ of $e$ as $f(e)=2(p(e)-n(e))+s(e)$, where we write $s(e)=\sum_{v\in V_2(D)\cap e}s(v)$.
}\end{Def}

\begin{lemma}\label{LEM_flaming}
If we apply R0--R3 moves for a spatial surface diagram, the framing of each edge does not change.
\end{lemma}

\begin{proof}
We can check easily that the framing of each edge does not change before and after an R0, R1, R2 or R3 move.
\end{proof}

\begin{rmk}\label{RMK:framing}{\rm
Applying an R4 move to a spatial surface diagram changes the framing of an edge.
There is a possibility that an R5 move changes the framing of an edge.
}\end{rmk}
We denote by R1$-$ the Reidemeister move R1 that reduces the number of crossings.

\begin{lemma}\label{LEM_R1_vanish}
Let $(D_i)_{1\le i\le n}$ be a sequence of spatial surface diagrams such that
\begin{itemize}
\item $D_1$ and $D_n$ have no bivalent vertices and
\item $D_i$ is transformed into $D_{i+1}$ by an R0 move or an R1$-$ move for any $i$.
\end{itemize}
Then, $D_1$ and $D_n$ are related by R${\omega}$ and R2--R3 moves in $\mathbb R^2$.
\end{lemma}

\begin{proof}
By our assumption, we see that $D_1$ is obtained by continuing to replace an local arc of $D_n$ with a teardrop-like piece, see Fig.~\ref{FIG:trivial_teardrop}.
On each edge of $D_1$, we move teardrop-like pieces into a small region so that all the teardrop-like pieces of an edge are aligned as depicted in Fig.~\ref{FIG:proof_aligned_teardrop}.
When a teardrop-like piece is passing through a crossing, R2--R3 moves are applied, see Fig.~\ref{FIG:teardrop_move}.
By the process, we have a new spatial surface diagram $D_1'$ such that $D_1'$ and $D_1$ are related by R2--R3 moves.
\par
We show that $D_n$ and $D_1'$ are related by R${\omega}$ and  R2--R3 moves; then the proof will be shown.
Since $D_1$ and $D_n$ are related by R0--R1 moves, $D_1'$ and $D_n$ are related by R0--R3 moves.
By Lemma~\ref{LEM_flaming}, the framing of each edge does not change before and after the process above.
Therefore, in each local area where the teardrop-like pieces are gathered, the number of positive self-crossings is equal to that of negative self-crossings.
In Fig.~\ref{FIG:proof_aligned_teardrop}, the B and C parts are canceled by R2--R3 moves by Lemma~\ref{LEM:easy_replacement}~(1), and the A and D parts are also canceled by R${\omega}$ moves by Lemma~\ref{LEM:easy_replacement}~(2).
Hence, $D_n$ and $D_1'$ are related by R${\omega}$ and  R2--R3 moves.
\end{proof}
		\begin{figure}[htbp]\centerline{
		\includegraphics[clip]{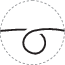}}
		\caption{A teardrop-like piece.}
		\label{FIG:trivial_teardrop}
		\end{figure}

		\begin{figure}[htbp]\centerline{
		\includegraphics{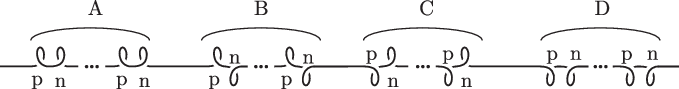}}
		\caption{The symbols ``p'' and ``n'' mean positive and negative self-crossings, respectively.}\label{FIG:proof_aligned_teardrop}
		\end{figure}

		\begin{figure}[htbp]\centerline{
		\includegraphics{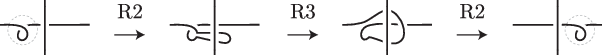}}
		\caption{Deformations by R2--R3 moves.}
		\label{FIG:teardrop_move}
		\end{figure}

The lemma below follows immediately.
\begin{lemma}\label{LEM_trivial}
Let $D$ and $D'$ be spatial surface diagrams.
We suppose that $D$ has no bivalent vertices.
If $D$ and $D'$ are related by R0, R2--R3 and R5 moves in $\mathbb R^2$, then $D$ and $\overline{D'}$ are related by R2--R3 and R5 moves in $\mathbb R^2$, and $\overline{D'}$ and $D'$ are related by R0 moves.
\end{lemma}

\begin{prop}\label{LEM:if_it_has_no_R4}
Let $D$ and $D'$ be spatial surface diagrams that have no bivalent vertices.
If $D$ and $D'$ are related by R0--R3 and R5 moves in $\mathbb R^2$, then $D$ and $D'$ are related by R${\omega}$, R2--R3, and R5 moves in $\mathbb R^2$.
\end{prop}

\begin{proof}
It suffices to construct a spatial surface diagram $D''$ with no bivalent vertices such that $D$ and $D''$ are related by R2--R3 and R5 moves and that $D''$ and $D'$ are related by R${\omega}$, R2--R3, and R5 moves.
\par
Since $D$ and $D'$ are related by R0--R3 and R5 moves, we fix a sequence $(D_i)_{1 \le i \le n}$ of spatial surface diagrams related by R0--R3 and R5 moves such that $D_1=D$ and $D_n=D'$.
Modifying the sequence $(D_i)_{1 \le i \le n}$, we will recursively construct a new sequence $(D_i')_{1 \le i \le n}$ of spatial surface diagrams satisfying that  $D_i'$ and $D_{i+1}'$ are related by R0, R2--R3 and R5 moves for every $i$ with $1\le i < n$.
We set $D_1'=D_1$.
Suppose that $D_i'$ is already defined.
\par
{\bf Case 1:} we consider the case where $D_{i}$ and $D_{i+1}$ are related by an R1 move.
We suppose that the left R1 move in Fig.~\ref{FIG:spatial_surface_Reidemeister_move} is applied.
As is illustrated in Fig.~\ref{FIG:proof_bivalent_small_disk}, we define $D_{i+1}'$ for two cases:
we attach a ``small bivalent disk'' on spatial surface diagrams, in stead of applying the left R1 move in Fig.~\ref{FIG:spatial_surface_Reidemeister_move}.
Each small bivalent disk contains one teardrop-like part and two bivalent vertices.
At a glance, $D_{i}'$ and $D_{i+1}'$ seem to be related by an R1 move.
However, $D_{i}'$ and $D_{i+1}'$ are actually related by R0 and R2 moves.
If a small bivalent disk attached to $D_i'$ is contained in the region where an R1 move is applied, we move it into on the outside of the region by using R0 and R2--R3 moves before we define $D_{i+1}'$, see Fig.~\ref{FIG:proof_small_R1_sliding}.
For also the right R1 move in Fig.~\ref{FIG:spatial_surface_Reidemeister_move}, we define $D_{i+1}'$ in the same manner.
\par
{\bf Case 2:} we consider the case where $D_{i}$ and $D_{i+1}$ are related by one of  R0, R2--R3 and R5 moves.
Let $\delta\subset\mathbb R^2$ be a disk in which the moves is applied.
If a small bivalent disk attached to $D_i'$ is contained in $\delta$, we move it on the outside of $\delta$ by using R0 and R2--R3 moves
before we define $D_{i+1}'$, as is the case 1.
We define $D_{i+1}'$ as follows:
$D_{i+1}'$ is equal to $D_i'$ in the outside of $\delta$ and the part of $D_{i+1}'$ contained in $\delta$ corresponds to the applied move.
		\begin{figure}[htbp]\centerline{
		\includegraphics{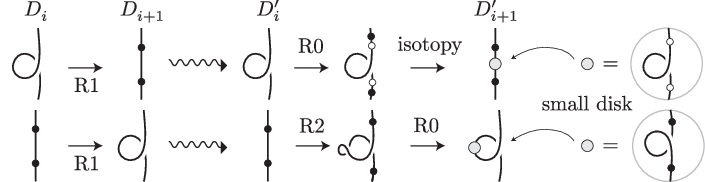}}
		\caption{Replace an R1 move with R0 and R2 moves.}
		\label{FIG:proof_bivalent_small_disk}
		\end{figure}

		\begin{figure}[htbp]\centerline{
		\includegraphics{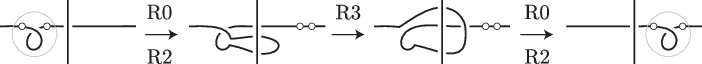}}
		\caption{This move is realized by R0 and R2--R3 moves.}
		\label{FIG:proof_small_R1_sliding}
		\end{figure}
\par
By the above definitions, we obtain a sequence $(D_i')_{1\le i\le n}$ such that $D_i'$ and $D_{i+1}'$ are related by R0, R2--R3 and R5 moves for any $i$.
Then, $D(=D_1'=D_1)$ and $D_n'$ are related by R0, R2--R3 and R5 moves.
Hence, by Lemma~\ref{LEM_trivial}, (i) the spatial surface diagrams $D$ and $D'':=\overline{D_n'}$ are related by R2--R3 and R5 moves, and $D''$ and $D_n'$ are related by R0 moves.
On the other hand, $D_n'$ is transformed into $D'(=D_n)$ by using R0 and R1$-$ moves in each small bivalent disk.
Then, $D''$ is also transformed into $D'$ by using R0 and R1$-$ moves.
By Lemma~\ref{LEM_R1_vanish}, we understand that (ii) $D''$ and $D'$ are related by R${\omega}$, R2--R3 and R5 moves in $\mathbb R^2$.
By using (i) and (ii), $D$ and $D'$ are related by R${\omega}$, R2--R3 and R5 moves.
\end{proof}

We denote by R4$-$ (resp.~R4$+$) the Reidemeister move R4 that reduces (resp.~increases) the number of crossings.
\begin{lemma}\label{LEM:no_R4}
Let $(D_i)_{1\le i\le n}$ be a sequence of spatial surface diagrams satisfying that
\begin{itemize}
\item $D_i$ is transformed into $D_{i+1}$ by an R4$-$ move for any $i$, and
\item R4$-$ moves are applied an even number of times at each trivalent vertex in $(D_i)_{1\le i\le n}$.
\end{itemize}
Then, $D_1$ and $D_n$ are related by R0--R3 and R5 moves in $\mathbb R^2$.
\end{lemma}
\begin{proof}
It is sufficient to prove that $D_1$ and $D_n$ are related by R0--R3 and R5 moves if R4$-$ moves are applied exactly two times at each vertex $v$ in $(D_i)_{1\le i\le n}$.
All situations around a trivalent vertex of $D_1$ are depicted in Fig.~\ref{FIG:proof_Y1}:
each local part in the figure is transformed into a Y-shaped part by using two R4$-$ moves.
		\begin{figure}[htbp]\centerline{
		\includegraphics{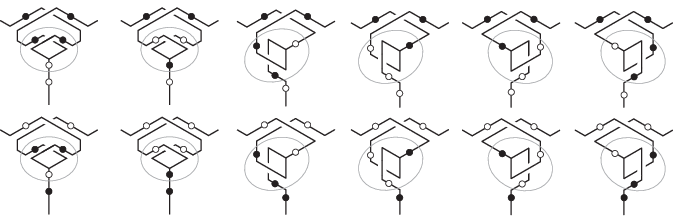}}
		\caption{All situations around a trivalent vertex.}
		\label{FIG:proof_Y1}
		\end{figure}
\par
As is depicted in Fig.~\ref{FIG:proof_R4_not_used}, the upper left of Fig.~\ref{FIG:proof_Y1} is deformed into a Y-shaped part by using R0--R3, R5 moves.
Similarly, each of Fig.~\ref{FIG:proof_Y1} is deformed into a Y-shaped part by using R0--R3, R5 moves, although its process is not depicted.
\end{proof}
		\begin{figure}[htbp]\centerline{
		\includegraphics{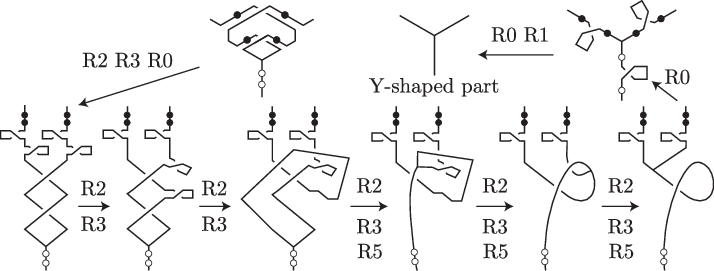}}
		\caption{A deformation into a Y-shaped part by R0--R3 and R5 moves.}
		\label{FIG:proof_R4_not_used}
		\end{figure}

\begin{prop}\label{LEM:main}
Let $(D_i)_{1\le i\le n}$ be a sequence of spatial surface diagrams satisfying that
\begin{itemize}
\item both $D_1$ and $D_n$ have no bivalent vertices,
\item $D_i$ and $D_{i+1}$ are related by one of R0--R5 moves for any $i$, and
\item R4 moves are applied an even number of times at each trivalent vertex in $(D_i)_{1\le i\le n}$.
\end{itemize}
Then, $D_1$ and $D_n$ are related by R${\omega}, $R2--R3 and R5  moves in $\mathbb R^2$.
\end{prop}

\begin{proof}
We will show that $D_1$ and $D_n$ are related by R0--R3 and R5 moves;
then $D_1$ and $D_n$ are related by R$\omega$, R2--R3 and R5 moves by Proposition~\ref{LEM:if_it_has_no_R4}, that is,
the proof will complete.
It suffices to construct a sequence $(D_i')_{1\le i\le n+m}$ of spatial surface diagrams with $D_1'=D_1$ and $D_{n+m}'=D_n$ satisfying that 
\begin{enumerate}
\item[(i)]
$D_i'$ and $D_{i+1}'$ are related by R0--R3 and R5 moves for any integer $1\le i<n$,
\item[(ii)]
$D_i'$ is transformed into $D_{i+1}'$ by an R4$-$ move for any integer $i$ with $n\le i<n+m$,
\item[(iii)]
R4$-$ moves are applied an even number of times at each trivalent vertex in $(D_i')_{n\le i\le n+m}$,
\end{enumerate}
since $D_n'$ and $D_{m+n}'$ are related by R0--R3 and R5 moves by Lemma~\ref{LEM:no_R4}.
\par
First, by modifying the given sequence $(D_i)_{1\le i\le n}$, we recursively construct a sequence $(D_i')_{1\le i\le n}$ of spatial surface diagrams.
We set $D_1':=D_1$. Suppose that $D_i'$ is already defined.
\par
{\bf Case 1:} we consider the case where $D_{i}$ and $D_{i+1}$ are related by an R4 move.
Suppose that the left R4 move in Fig.~\ref{FIG:spatial_surface_Reidemeister_move} is applied.
Let $v$ be the trivalent vertex in which the move is applied.
We define $D_{i+1}'$ as illustrated in Fig.~\ref{FIG:proof_R4_trivalent_disk}:
we replace $v$ in $D_{i}'$ with a ``small trivalent disk'', instead of applying the left R4 move in Fig.~\ref{FIG:proof_R4_trivalent_disk}.
Each small trivalent disk contains an trivalent vertex and three bivalent vertices.
At a glance, $D_{i}'$ and $D_{i+1}'$ seem to be related by an R4 move.
However, $D_{i}'$ and $D_{i+1}'$ are actually related by R0 and R2 moves.
If  a trivalent vertex is already replaced with a small trivalent disk, we think of the trivalent disk as a trivalent vertex.
In the case, the small trivalent disk contains a smaller trivalent disk.
For the case where the right R4 move in Fig.~\ref{FIG:spatial_surface_Reidemeister_move} is applied, we define $D_{i+1}'$ in the same manner.
		\begin{figure}[htbp]\centerline{
		\includegraphics{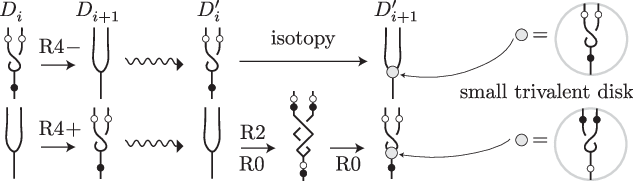}}
		\caption{A trivalent vertex is replaced by a ``small trivalent disk''.}
		\label{FIG:proof_R4_trivalent_disk}
		\end{figure}
\par
{\bf Case2:} we consider the case where $D_{i}$ and $D_{i+1}$ are related by an R0--R3 or R5 move.
Suppose that an R0, R1, R2 or R3 move is applied in a region $\delta$.
We define the spatial surface diagram $D_{i+1}'$ as follows:
$D_{i+1}'$ is equal to $D_i'$ on the outside of $\delta$, and $D_{i+1}'$ is equal to $D_{i+1}$ on the inside of $\delta$.
Suppose that an R5 move is applied in a region $\delta$.
If $\delta$ contains no trivalent small disk, we define $D_{i+1}'$ in the same way above.
If $\delta$ contains trivalent small disks, we regard the outermost small trivalent disk as a trivalent vertex and
define $D_{i+1}'$ in the same way above.
In this case, $D_{i+1}'$ and $D_{i}'$ are related by R0, R2--R3 and R5 moves, see Fig.~\ref{FIG:proof_R4_sliding}.
		\begin{figure}[htbp]\centerline{
		\includegraphics{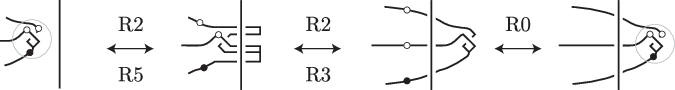}}
		\caption{An R5 move at a small trivalent disk is realized by R0, R2--R3 and R5 moves.}
		\label{FIG:proof_R4_sliding}
		\end{figure}
\par
By the construction above, (i) holds.
\par
Secondly, we construct a sequence $(D_i')_{n\le i\le n+m}$ of spatial surface diagrams.
We continue applying R4$-$ moves in order starting from the innermost small trivalent disks of $D_n'$, until trivalent disks disappear.
We have $D_n$, finally.
Let $(D_i')_{n\le i\le n+m}$ be a sequence obtained by the above process satisfying that $D_{n+m}'=D_n$ and the condition (ii).
The number of R4 moves in $(D_i)_{1\le i\le n}$ is equal to the number of R4$-$ moves in $(D_i')_{n\le i\le n+m}$ at each trivalent vertex.
Since R4 moves are applied an even number of times at each trivalent vertex in $(D_i)_{1\le i\le n}$, then the condition (iii) holds.
\end{proof}

\begin{proof}[Proof of Theorem~\ref{THM:main_oriented}.]
Since $F$ and $\overrightarrow{\operatorname{Sf}}(D)$ are ambient isotopic, we assume that $F=\overrightarrow{\operatorname{Sf}}(D)$.
For the same reason, we assume that $F'=\overrightarrow{\operatorname{Sf}}(D')$.
\par
We suppose that (b) holds; we show that (c) holds.
By Lemma~\ref{LEM:easy_replacement}~(2), an R${\omega}$ move is realized by R2--R3 and R5 moves in $S^2$; then (c) holds.
\par
We suppose that (c) holds; we show that (a) holds.
If $D$ and $D'$ are related by exactly one of R2--R3 and R5--R6 moves, we can immediately construct an isotopy between $F$ and $F'$ that preserves the orientations of $F$ and $F'$.
Since $D$ and $D'$ are related by finitely many R2--R3 and R5--R6 moves, then $F$ and $F'$ are also ambient isotopic including orientations: (a) holds.
\par
We suppose that (a) holds; we show that (b) holds.
Let $G$ (resp.~$G'$) be the trivalent spine of $F$ (resp.~$F'$) such that $D(G;F)=D$ (resp.~$D(G';F')=D'$).
Since $F$ and $F'$ are ambient isotopic including orientations, we fix an isotopy $\{h_t:\mathbb R^3\to\mathbb R^3\}_{t\in[0,1]}$
with $h_0=\mathrm{id}|_{\mathbb R^3}$ such that $h_1(F)=F'$ including orientations, and that $(h_1(G);h_1(F))$ is in regular position.
By Lemma~\ref{LEM:Step2(Luo_spatial_surface_version)}~(2), $D(h_1(G);h_1(F))$ and $D'$ are related by R2--R3 and R5--R6 moves in $\mathbb R^2$.
Then, it suffices to show that $D$ and $D(h_1(G);h_1(F))$ are related by R${\omega}$, R2--R3 and R5 moves in $\mathbb R^2$.
By Lemma~\ref{LEM:step1(simple_spine_into_complicated_spine)}, $D$ and $D(h_1(G);h_1(F))$ are related by R0--R5 moves.
Let $(D_i)_{1\le i\le n}$ be a sequence of spatial surface diagrams with $D_1=D$ and $D_n=D(h_1(G);h_1(F))$ such that
$D_i$ and $D_{i+1}$ are related by one of R0--R5 moves for any $i$, where we note that both $D_1$ and $D_n$ have no bivalent vertices.
Since the orientation of the oriented spatial surface $F$ is equal to that of the oriented spatial surface $F'$, then R4 moves are applied an even number of times at each trivalent vertex in $(D_i)_{1\le i\le n}$.
Hence, $(D_i)_{1\le i\le n}$ satisfies all of the conditions in Proposition~\ref{LEM:main}.
By Proposition~\ref{LEM:main}, $D_1$ and $D_n$ are related by R${\omega}$, R2--R3 and R5 moves.
\end{proof}

\section{A diagram of a non-split spatial surface}\label{SEC:non_split_spatial_surface}
In Sections~\ref{SEC:spatial_surface} and \ref{SEC:proof}, we called a spatial surface with boundary a spatial surface for short.
In this section, we use the notion of a spatial surface in the original definition:
a spatial surface may have some closed components, see Section~\ref{SEC:spatial_surface}.
We assume that a spatial surface has no sphere components and no disk components, since a sphere or disk in $S^3$ is unique up to ambient isotopy.
\par
A spatial surface $S$ is {\it split} if there is a 2-sphere $P$ embedded in $S^3\setminus\operatorname{int}S$ such that $P$ bounds no balls in $S^3\setminus\operatorname{int}S$.
A spatial surface $S$ is {\it non-split} if $S$ is not split, that is, any sphere embedded in $S^3\setminus \operatorname{int}S$ bounds a ball.
We note that any 1-component spatial surface is always non-split.
\par
In this section, we consider a way to present non-split spatial surfaces.
Let $S$ be a non-split spatial surface.
If $S$ has closed components $S_1$, \ldots, $S_n$, we remove the interior of a disk $\delta_i$ from $S_i$.
By the finitely many operations of removing a disk interior in each closed component, we have a spatial surface $F:=S\setminus \operatorname{int}\delta$ with boundary, where we put $\delta=\delta_1\sqcup\cdots\sqcup\delta_n$.
Proposition~\ref{PROP:many_hole_reconstruct_unique} claims that $F$ has necessary and sufficient information of the original spatial surface $S$ up to ambient isotopy.
Therefore, when we deal with a spatial surface that has closed components, it is sufficient to consider a spatial surface with boundary that is obtained from the original spatial surface by removing a disk interior of each closed component.
\begin{prop}\label{PROP:many_hole_reconstruct_unique}
Let $S$ be a non-split spatial surface, and let $\delta:=\delta_1\sqcup\cdots\sqcup\delta_n$ be a disjoint union of disks in $\operatorname{int}S$.
We denote by $F$ the non-split spatial surface $S\setminus\operatorname{int}\delta$.
Let $\Delta:=\Delta_1\sqcup\cdots\sqcup\Delta_n$ be a disjoint union of disks  in  $S^3\setminus\operatorname{int}F$ such that $\partial \Delta_i=\partial\delta_i$ for any integer $i$ with $1\le i\le n$.
Then, the spatial surfaces $S$ and $F\cup\Delta$ are ambient isotopic in $S^3$.
\end{prop}
\begin{proof}
We use the Cut-And-Paste method.
By an isotopy of $S^3\setminus\operatorname{int}F$, we deform $\Delta$ so that $|L|$ is minimal, where we put $L=\operatorname{int}\delta\cap\operatorname{int}\Delta$, and $|L|$ is the number of connected components of $L$.
\par
We show that $L=\emptyset$; we suppose that $L \not= \emptyset$.
We may assume that every connected component of $L$ is a loop.
We take an innermost disk $\delta_0$ on $\delta$, where we note $\Delta\cap\operatorname{int}\delta_0=\emptyset$ and $\partial\delta_0$ is a loop in $L$.
Put $l_0=\partial\delta_0$.
Let $\Delta_0^+$ be the disk on $\Delta$ such that $\partial\Delta_0^+=l_0$.
Let $A$ be a thin regular neighborhood of $l_0$ in $\Delta$, where we note that $A$ is an annulus in $\Delta$.
Let $l^+$ (resp.~$l^-$) be the loop of $\partial A$ such that $l^+\subset\Delta_0^+$ (resp.~$l^-\subset\Delta\setminus\Delta_0^+$).
Let $\delta^+$ (resp.~$\delta^-$) be a disk in $S^3\setminus S$ such that $\delta^+$ (resp.~$\delta^-$) is parallel to $\delta_0$ and $\partial \delta^+=l^+$ (resp.~$\partial\delta^-=l^-$).
Put $\Delta^-=(\Delta\setminus\Delta_0^+\setminus A)\cup\delta^-$.
Since $F$ is non-split, the sphere $(\Delta_0^+\setminus A)\cup\delta^+$ bounds a ball in $S^3\setminus\operatorname{int}F$.
Then, $F\cup\Delta$ and $F\cup\Delta^-$ are ambient isotopic in $S^3\setminus\operatorname{int}F$.
On the other hand, $|L'|$ is less than $|L|$, where we put $L'=\operatorname{int}\Delta^-\cap\operatorname{int}\delta$.
This leads to a contradiction with the minimality of $|L|$.
Hence, $L=\emptyset$.
\par
We put $P_i=\delta_i\cup\Delta_i$ for any integer $i$ with $1\le i\le n$, where we note that $\Delta\cup\delta=P_1\sqcup\cdots\sqcup P_n$ is a disjoint union of spheres embedded in $S^3 \setminus\operatorname{int}F$.
Since $F$ is non-split, each sphere $P_i$ bounds a ball $B_i$ in $S^3\setminus\operatorname{int}F$; then $F\cap \operatorname{int}B_i =\emptyset$.
Let $i$ and $j$ be integers with $i \not=j$.
If $B_j \subset \operatorname{int}B_i$, it holds that $\partial\delta_j\subset P_j=\partial B_j\subset\operatorname{int}B_i$, and we have $F \cap \operatorname{int}B_i \not= \emptyset$, since $\partial\delta_j\subset F$. This leads to a contradiction.
Hence, we have $B_j\not\subset\operatorname{int}B_i$, that is, $B_i\cap B_j=\emptyset$;
therefore the balls $B_1$, \ldots, $B_n$ are pairwise disjoint.
\par
In each ball $B_i$, we deform $\Delta_i$ into $\delta_i$ by an isotopy of $S^3\setminus\operatorname{int}F$.
Therefore, the spatial surfaces $F\cup\delta$ and $F\cup\Delta$ are ambient isotopic in $S^3$, where we note that $F\cup\delta=S$.
\end{proof}

In Proposition~\ref{PROP:many_hole_reconstruct_unique}, if $S$ is split, the claim does not always hold because there is no information about a ``partition''  of $S^3$ by the closed components of $S$.

\begin{thm}\label{THM_spatial_surface_with/without_boundary}
Let $S$ (resp.~$S'$) be a non-split unoriented spatial surface.
Let $F$ (resp.~$F'$) be the non-split unoriented spatial surface with boundary that is obtained from $S$ (resp.~$S'$) by removing an open disk from every closed component.
Let $D$ (resp.~$D'$) be a diagram of $F$ (resp.~$F'$).
Then, the following conditions are equivalent.
\begin{enumerate}
\item[(a)] $S$ and $S'$ are ambient isotopic in $S^3$.
\item[(b)] $D$ and $D'$ are related by R0--R6 moves in $\mathbb{R}^2$.
\end{enumerate}
\end{thm}

\begin{proof}
Suppose that (a) is satisfied.
We show that (b) holds.
Let $\delta$ (resp.~$\delta'$) be a disjoint union of disks contained in $S$ (resp.~$S'$) such that $F=S \setminus \operatorname{int}(\delta)$ (resp.~$F'=S'\setminus\operatorname{int}(\delta')$).
Since $S$ and $S'$ are ambient isotopic in $S^3$,
we take an isotopy $\{h_t:S^3\to S^3\}_{t \in [0,1]}$ such that $h_0=\mathrm{id}|_{S^3}$ and $h_1(S)=S'$.
Then, $F$ and $S'\setminus\operatorname{int}h_1(\delta)$ are ambient isotopic in $S^3$.
On the other hand, since $h_1(\delta)$ and $\delta'$ are ambient isotopic in $S'$,
the spatial surfaces $S'\setminus\operatorname{int}h_1(\delta)$ and $F'$ are ambient isotopic in $S'$;
then $S'\setminus\operatorname{int}h_1(\delta)$ and $F'$ are also ambient isotopic in $S^3$.
Therefore, $F$ and $F'$ are ambient isotopic in $S^3$.
By Theorem~\ref{THM:main_non_oriented}, $D$ and $D'$ are related by R0--R6 moves in $\mathbb{R}^2$.
\par
Suppose that (b) is satisfied.
By Theorem~\ref{THM:main_non_oriented}, $F$ and $F'$ are ambient isotopic in $S^3$.
Since $F$ and $F'$ are non-split, $S$ and $S'$ are ambient isotopic in $S^3$ by Proposition~\ref{PROP:many_hole_reconstruct_unique}: (a) is satisfied.
\end{proof}
In the same manner, we have Theorem~\ref{THM_spatial_surface_with/without_boundary_2} below by using Theorem~\ref{THM:main_oriented} and Proposition~\ref{PROP:many_hole_reconstruct_unique}.
\begin{thm}\label{THM_spatial_surface_with/without_boundary_2}
Let $S$ (resp.~$S'$) be a non-split oriented spatial surface.
Let $F$ (resp.~$F'$) be the non-split oriented spatial surface with boundary that is obtained from $S$ (resp.~$S'$) by removing an open disk from every closed component.
Let $D$ (resp.~$D'$) be a diagram of $F$ (resp.~$F'$).
Then, the following conditions are equivalent.
\begin{enumerate}
\item[(a)] $S$ and $S'$ are ambient isotopic in $S^3$ including orientations.
\item[(b)] $D$ and $D'$ are related by R${\omega}$, R2--R3 and R5--R6 moves in $\mathbb R^2$.
\item[(c)] $D$ and $D'$ are related by R2--R3 and R5--R6 moves in $S^2$.
\end{enumerate}
\end{thm}

Next, we explain relations among knots, handlebody-knots and spatial surfaces.
For a manifold $X$ embedded in $S^3$, we denote by $[X]$ the ambient isotopy class of $X$ in $S^3$.
We put
\begin{align*}
\mathcal K&=\{[K]\mid\text{$K$ is a non-split link}\},\\
\mathcal H&=\{[H]\mid\text{$H$ is a non-split handlebody-link}\},\\
\mathcal {S}_{\rm cl}&=\{[S]\mid\text{$S$ is a non-split spatial closed surface}\},\\
\mathcal F_{\rm or}&=\{[F]\mid\text{$F$ is a non-split orientable spatial surface with boundary}\},
\end{align*}
where we note that a knot is a kind of non-split link, and a handlebody-knot is kind of non-split handlebody-link.
The following maps are injective:
\begin{align*}
f_1:\mathcal K\to\mathcal H;[K]\mapsto[N_K], & &
f_2:\mathcal H\to\mathcal S_{\rm cl};[H]\mapsto[\partial H], & &
f_3:\mathcal S_{\rm cl}\to\mathcal F_{\rm or};[S]\mapsto[F_S],
\end{align*}
where $N_K$ means a regular neighborhood of $K$ in $S^3$, and $F_S$ means a spatial surface with boundary obtained from $S$ by removing the interior of a disk in each closed component of $S$.
Injectivity of $f_3$ follows from Proposition~\ref{PROP:many_hole_reconstruct_unique}.
\par
By injectivity of the above maps, we have new presentation of non-split links and non-split handlebody-links
by using spatial surface diagrams, see Fig.~\ref{FIG:knot_to_application}.
This suggests a new approach to studying a knot, link, handlebody-knot and handlebody-link in a framework of spatial surfaces.
		\begin{figure}[htbp]\centerline{
		\includegraphics{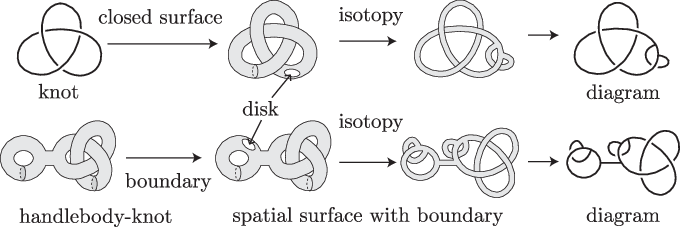}}
		\caption{Presentation of knots and handlebody-knots by using spatial surface diagrams.}\label{FIG:knot_to_application}
		\end{figure}

\begin{rmk}{\rm
The map $f_2$ above is not surjective.
There are infinitely many spatial closed surfaces that bounds no handlebody-knots up to ambient isotopy.
The spatial closed surface $S$ of Fig.~\ref{FIG:handlebody_not_bound} bounds no handlebodies:
each fundamental group of connected components of $S^3\setminus S$ is not free (cf.~\cite{homma1954}).
}\end{rmk}
		\begin{figure}[htbp]\centerline{
		\includegraphics{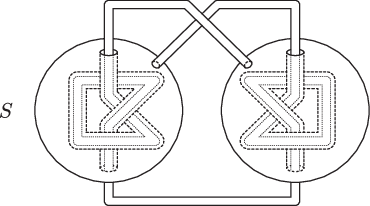}}
		\caption{A spatial closed surface with genus 2 that bounds no handlebody-knots.}
		\label{FIG:handlebody_not_bound}
		\end{figure}

For any component of a diagram $D$ of a link, which has no trivalent vertices, we attach an ear-like edge as illustrated in Fig.~\ref{FIG:ear}.
The diagram, which has trivalent vertices, is denoted by $D_{\varphi}$.
\par
By using Proposition~\ref{PROP:many_hole_reconstruct_unique}, we have the following Theorem.

\begin{thm}\label{THM_knot_new_Reidemeister_move}
Let $D$ and $D'$ be diagrams of non-split unoriented links $K$ and $K'$, respectively.
Then, the following conditions are equivalent.
\begin{itemize}
\item[(a)] $K$ and $K'$ are ambient isotopic in $S^3$.
\item[(b)] $D_{\varphi}$ and $D_{\varphi}'$ are related by R2--R3 and R5--R6 moves.
\end{itemize}
\end{thm}
		\begin{figure}[htbp]\centerline{
		\includegraphics{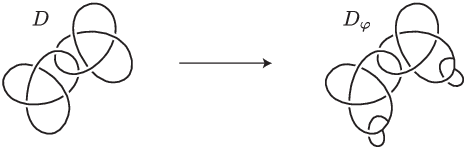}}
		\caption{A diagram $D$ of a link and $D_\varphi$.}
		\label{FIG:ear}
		\end{figure}

\section*{Acknowledgements}
The author would like to thank the referee, Atsushi ISHII, Koki TANIYAMA and Tomo MURAO for valuable discussions and suggestions.

\end{document}